\documentclass[12pt]{amsart}

\usepackage{microtype} 

\usepackage[left=3cm,right=3cm,bottom=2cm,top=3cm]{geometry}
\usepackage{amsmath}
\usepackage{amssymb}
\usepackage{mathrsfs}
\usepackage[all]{xy}
\usepackage{graphicx}
\usepackage{latexsym}
\usepackage{graphics} 
\usepackage{subfigure}
\usepackage{psfrag}  
\usepackage{enumerate}
\usepackage{pstricks,pst-plot,pst-node}

\theoremstyle{plain}

\newtheorem{thm}{Theorem}[section]
\newtheorem{cor}[thm]{Corollary}
\newtheorem{lemma}[thm]{Lemma}
\newtheorem{prop}[thm]{Proposition}
\theoremstyle{definition}
\newtheorem{definition}[thm]{Definition}
\newtheorem{remark}[thm]{Remark}
\newtheorem{ex}[thm]{Example}
\theoremstyle{conjecture}

\def\B{\mathcal{B}}
\def\C{\mathcal{C}}
\def\D{\mathcal{D}}

\def\I{\mathcal{I}}

\def\P{\mathcal{P}}

\def\R{\mathcal{R}}

\def\T{\mathcal{T}}

\def\sfQ{\mathsf{Q}}
\def\sfA{\mathsf{A}}
\def\sfG{\mathsf{G}}
\def\sfS{\mathsf{S}}
\def\Si{\mathsf{Si}}
\def\So{\mathsf{So}}
\def\I{\mathsf{Iso}}

\def\Hom{\operatorname{Hom}}
\def\End{\operatorname{End}}
\def\Ext{\operatorname{Ext}}
\def\add{\operatorname{add}}
\def\ind{\operatorname{ind}}

\begin{document}

\title{Maximal rigid objects as noncrossing bipartite graphs}

\author{Raquel Coelho Sim\~oes}
\address{School of Mathematics \\
University of Leeds \\ Leeds LS2 9JT, UK.}
\email{rcsimoes@maths.leeds.ac.uk}
\date{\today}

\begin{abstract}
We classify the maximal rigid objects of the $\Sigma^2 \tau$-orbit category $\C(Q)$ of the bounded derived category for the path algebra associated to a Dynkin quiver $Q$ of type $A$, where $\tau$ denotes the Auslander-Reiten translation and $\Sigma^2$ denotes the square of the shift functor, in terms of bipartite noncrossing graphs (with loops) in a circle. We describe the endomorphism algebras of the maximal rigid objects, and we prove that a certain class of these algebras are iterated tilted algebras of type $A$.
\end{abstract}




\maketitle

\bibliographystyle{plain}

\section{Introduction}

Let $Q$ be a Dynkin quiver of type $A_n$, and $\D^b(Q)$ the bounded derived category of the path algebra associated to $Q$, with shift functor $\Sigma$ and Auslander-Reiten translation $\tau$. 

The setting of this paper is the orbit category $\C(Q) := \D^b(Q) / \Sigma^2 \tau$. The motivation to work in this category comes from work done by Riedtmann \cite{Riedtmann} in order to classify self-injective algebras of finite representation type. Riedtmann shows that maximal Hom-free objects, also known as Hom-configurations, in the bounded derived category are invariant under the functor $\Sigma^2 \tau$. Riedtmann also proves that, in type $A_n$, this class of objects is in bijection with the set of classical noncrossing partitions of an $n$-gon. This result was generalized to any simply-laced Dynkin type in \cite{S} (see also \cite{BRT}).

 We remark that it follows from \cite[Theorem 7.0.5]{A} that the orbit category $\C(Q)$ is triangle equivalent to the stable module category of the selfinjective Nakayama algebra $\sfA = K \sfQ / I$, where $\sfQ$ is an $n$-cycle, and $I = R^{n+1}$, where $R$ denotes the arrow ideal of $K\sfQ$.

Maximal rigid objects, which coincide with Ext-configurations or cluster-tilting objects in the cluster category, play a key role in tilting and cluster-tilting theory (for more details see, for example, the surveys \cite{BM, Kellersurvey, Reitensurvey} and the book \cite{htt}). This suggests that the study of this class of objects in $\C(Q)$ might be worthwhile. In the cluster category of type $A_n$, maximal rigid objects have a nice geometrical characterization given by triangulations of a regular ($n+3$)-gon (cf. \cite{BMRRT, CCS}). We aim to give a combinatorial characterization of maximal rigid objects in $\C(Q)$ as well.

In order to do this, we give a geometric model for the category $\C(Q)$, inspired by Riedtmann's bijection mentioned above. 

Given that, in $\C(Q)$, the Hom-vanishing objects have a better behaviour than that of the Ext-vanishing objects, we do not get such a neat characterization for the maximal rigid objects in $\C(Q)$. However, we still have an interesting relationship with other simple combinatorial objects: the noncrossing bipartite graphs.

In \cite{CCS} cluster tilted algebras of type $A_n$, which are the endomorphism algebras of the cluster-tilting objects in type $A_n$, were described in terms of quivers with relations using triangulations of a regular ($n+3$)-gon. We will see that the description of the endomorphism algebras of the maximal rigid objects in $\C(Q)$ in terms of quivers with relations has some similarities. We will also see that the subclass of these endomorphism algebras for which the quiver does not have any cycle are iterated-tilted algebras of type $A_n$. It is interesting to compare this result with Happel \cite[Theorem 1]{Happelcorrected} which states that the endomorphism algebra of a tilting set in $\D^b(Q)/ \Sigma^2$ which is simply connected must be an iterated tilted algebra.   

This paper is organized as follows. In Section \ref{secgeomodel} we give a geometric model of $\C(Q)$ and in Section \ref{sechomext} we read off morphisms and extensions between indecomposable objects from the geometric model. 

In Section \ref{secclass} we give a characterization of maximal rigid objects in $\C(Q)$. We will see that these can be described as certain bipartite noncrossing graphs (which might contain isolated vertices) together with loops satisfying certain conditions. 

The strategy is to separate the following components of these graphs: the oriented edges between distinct vertices, the isolated vertices and the loops. 

We show that given a maximal rigid object, if we remove all its loops, we get a loop-free maximal rigid object, meaning that removing loops does not allow us to add arrows between distinct vertices while preserving the Ext-free property. This fact allows us to ``ignore'' the loops and study them after characterizing loop-free maximal rigid objects. These can be described as certain \textit{tilings} of $\P_n$ with isolated vertices satisfying certain properties. 

In Section \ref{secendalg} we describe the endomorphism algebras of these maximal rigid objects in terms of quivers with relations using the geometric characterization we have given in Section 3. Section \ref{seciterated} is dedicated to the relationship between these endomorphism algebras and the iterated tilted algebras of type $A_n$. 



\section{A geometric model of $\C(Q)$ in type $A$}\label{secgeomodel}

Firstly we need to fix some notation that will be used throughout this paper.

Let $K$ be an algebraically closed field, $n$ an integer and let $Q$ a quiver of type $A_n$. We denote by $KQ$ the corresponding path algebra and by $\D^b (KQ)$ the bounded derived category of modules over the path algebra $KQ$. Note that all the modules considered will be left finite dimensional modules. Let $\tau$ be the AR-translate on $\D^b (KQ)$ and $\Sigma$ the shift functor . Let $\C(Q)$ denote the orbit category of $\D^b (KQ)$ by $\tau \Sigma^2$.

We denote by $\P_n$ a regular $n$-gon with vertices $1, \ldots, n$ numbered clockwise around the boundary. We will consider $\P_n$ as a disk with $n$ marked points on the boundary, numbered clockwise.

We will need some notions and notations from graph theory. Let $\sfG$ be a directed graph, where loops are allowed.

Given an arrow $\alpha$ in $\sfG$, we denote by $s(\alpha)$ (respectively, $t(\alpha)$) the origin (respectively, target) of $\alpha$. Given a vertex $i$ of $\sfG$, we denote by $v(i)$ the valency of $i$, which is the number of arrows incident with $i$.

We say that a vertex $i$ is:
\begin{enumerate}
\item an \textit{isolated vertex} if there are no arrows (including loops) incident with it,
\item a \textit{source} (respectively, \textit{sink}) if $i$ is not isolated and given an arrow $\alpha$ incident with $i$, we have $s(\alpha) = i$ (respectively, $t(\alpha) = i$).
\end{enumerate}
We denote the set of sources by $\So$, the set of sinks by $\Si$ and the set of isolated vertices by $\I$. Note that these sets are disjoint, i.e., isolated vertices are not considered to be either in $\Si$ or $\So$. 

\textbf{Notation:} In the figures appearing in this paper, the sources are denoted by empty circles $\circ$, the sinks by filled circles $\bullet$ and the isolated vertices by $\times$.

\begin{ex}
In the graph of Figure \ref{figuressiexample}, we have $2 \in \I, 1 \in \So, 3, 5 \in \Si$ and $4 \not\in \I, \So, \Si$.

\begin{figure}[!ht]
\includegraphics[scale=.8]{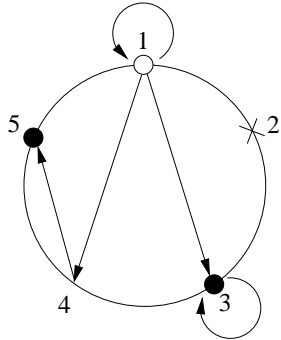}
\caption{Sources, sinks and isolated vertices.}
\label{figuressiexample}
\end{figure}
\end{ex}

In this section we give a geometric construction of $\C(Q)$, motivated by Riedtmann's bijection between noncrossing partitions and Hom-configurations (cf. \cite[2.6]{Riedtmann}). We will use the same method as that in \cite{BaurMarsh}, where the authors give a geometric model for the so called root category (also for type $A_n$), and so the proofs will be omitted. We consider oriented edges between vertices of a regular $n$-gon $\P_n$. Boundary edges and loops are included, and we denote the edge oriented from $i$ to $j$ by $[i,j]$.

Let $\Gamma (n) = \Gamma$ be the quiver defined as follows: the vertices are the set of all possible oriented edges between vertices of $\P_n$, including loops. The set of vertices of $\Gamma$ will be denoted by $\Gamma_0$. The arrows are of the form $[i,j] \rightarrow [i+1,j]$, for $j \ne i+1$, and $[i,j] \rightarrow [i,j+1]$, for $i \ne j$, where $i+1, j+1$ are taken modulo $n$.

Let $\tau$ be the automorphism of $\Gamma (n)$ obtained by rotating edges in $\P_n$ through $2 \pi /n$ anticlockwise; thus $\tau ([i,j]) = [i-1, j-1]$.

\begin{lemma}
The pair $(\Gamma, \tau)$ is a stable translation quiver.
\end{lemma}
%

\begin{ex}
We consider the case when $n= 5$, so $\P_n$ is a pentagon. The translation quiver $\Gamma (5)$ is given in the following figure.

\[
\xymatrix@C=0.1cm{& & & & 11 \ar@{->}[dr] & & 22 \ar@{->}[dr] & & 33 \ar@{->}[dr] & & 44 \ar@{->}[dr] & & 55 \ar@{->}[dr] & & 11 \\
& & & 15 \ar@{->}[dr]\ar@{->}[ur] & & 21 \ar@{->}[dr]\ar@{->}[ur] & & 32 \ar@{->}[dr]\ar@{->}[ur] & & 43 \ar@{->}[dr]\ar@{->}[ur] & & 54 \ar@{->}[dr]\ar@{->}[ur] & & 15 \ar@{->}[ur] \\
& & 14 \ar@{->}[dr]\ar@{->}[ur] & & 25 \ar@{->}[dr]\ar@{->}[ur] & & 31 \ar@{->}[dr]\ar@{->}[ur] & & 42 \ar@{->}[dr]\ar@{->}[ur] & & 53 \ar@{->}[dr]\ar@{->}[ur] & & 14 \ar@{->}[ur]\\
& 13 \ar@{->}[dr]\ar@{->}[ur] & & 24 \ar@{->}[dr]\ar@{->}[ur] & & 35 \ar@{->}[dr]\ar@{->}[ur] & & 41 \ar@{->}[dr]\ar@{->}[ur] & & 52 \ar@{->}[dr]\ar@{->}[ur] & & 13 \ar@{->}[ur]\\
12 \ar@{->}[ur] & & 23 \ar@{->}[ur] & & 34 \ar@{->}[ur] & & 45 \ar@{->}[ur] & & 51 \ar@{->}[ur] & & 12 \ar@{->}[ur]}
\]
\end{ex}

Note that in the figure above $ij = [i,j]$.

By \cite{Keller}, $\C(Q)$ is a triangulated category, and by \cite{BMRRT}, it has AR-triangles and its AR-quiver, $\Gamma (\C(Q))$, is the quotient of the AR-quiver of $\D^b (KQ)$ by the automorphism induced by $\tau \Sigma^2$. We have the following:

\begin{prop}\label{AR-quiver}
The translation quiver $\Gamma (n)$ is isomorphic to $\Gamma (\C(Q))$.
\end{prop}

We note that the category $\C(Q)$ is standard. We thus have the following corollary of \ref{AR-quiver}, giving a geometric realization of $\C(Q)$. 

\begin{cor}
The orbit category $\C(Q)$ is equivalent to the additive hull of the mesh category of $\Gamma (n)$. 
\end{cor}

We shall identify indecomposable objects in $\C(Q)$, up to isomorphism, with the corresponding oriented edges between vertices of $\P_n$, and we shall freely switch between objects and oriented edges. 

%
%

\section{Hom and Ext-groups in the geometric model}\label{sechomext}

In this section we indicate how morphisms and extensions between indecomposable objects in $\C(Q)$ can be read off from the geometric model. Firstly we need to fix the following notation. Given the vertices $i_1, i_2, \ldots, i_k$ of $\P_n$, we write $C(i_1, i_2, \ldots, i_k)$ to mean that $i_1, i_2, \ldots, i_k, i_1$ follow each other under the clockwise circular order on the boundary.

Fix $i, j$ with $1 \leq i, j \leq n$, and consider the corresponding indecomposable object $[i,j]$. Then we have:

\begin{lemma}\label{Rectangles}
Let $X, Y$ be indecomposable objects in $\C(Q)$. Then:
\begin{enumerate}
\item $\Hom_{\C(Q)} ([i,j],Y) \ne 0$ if and only if $Y \in \R_F$, where $\R_F$ is the rectangle with corners: $[i,j], [i,i], [j-1,j], [j-1,i]$.
\item $\Hom_{\C(Q)} (X, [i,j]) \ne 0$ if and only if $X \in \R_B$, where $\R_B$ is the rectangle with corners: $[j,i+1], [j,j], [i,i+1], [i,j]$.
\end{enumerate}
\end{lemma}


\begin{figure}[!ht]
\psfragscanon
\psfrag{R_F}{$\R_F$}
\psfrag{R_B}{$\R_B$}
\psfrag{[i,i]}{$[i,i]$}
\psfrag{[j-1,i]}{$[j-1,i]$}
\psfrag{[j-1,j]}{$[j-1,j]$}
\psfrag{[i,j]}{$[i,j]$}
\psfrag{[i,i+1]}{$[i,i+1]$}
\psfrag{[j,i+1]}{$[j,i+1]$}
\psfrag{[j,j]}{$[j,j]$}
\includegraphics[scale=.8]{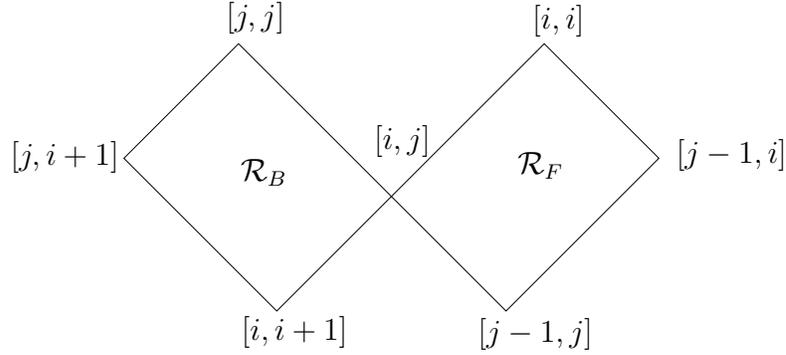}
\caption{The Hom-hammocks associated to $[i,j]$.}
\label{fig2}
\end{figure}

\textbf{Note:} If $\Hom_{\C(Q)} ([i,j], Y) \ne 0$ (respectively, $\Hom_{\C(Q)} (X, [i,j]) \ne 0$) then it is one-dimensional.

Lemma \ref{Rectangles} can be reinterpreted geometrically as follows:

\begin{prop}\label{readingHoms}
Let $X, Y$ be indecomposable objects in $\C(Q)$. Then $\Hom_{\C(Q)} (X,Y) \ne 0$ if and only if one of the following conditions holds: 
\begin{enumerate}
\item $X$ and $Y$ are two distinct arrows, with $s(X) \ne t(X)$ and $s(Y) \ne t(Y)$ such that $C(s(X), s(Y), t(X), t(Y))$.
\item $X$ and $Y$ have the same source, and $C(s(Y) = s(X), t(X), t(Y))$. 
\item $X$ and $Y$ have the same target, and $C(s(X), s(Y), t(X) = t(Y))$.
\item $s(X) = t(Y)$ and $C(t(X), s(X) = t(Y), s(Y))$.
\end{enumerate}
Note that in cases (2), (3) and (4), $X$ or $Y$ can be loops. 
\end{prop}


\begin{remark}
In the light of \ref{readingHoms}, it is now clear how the geometric model of $C(Q)$ is motivated by Riedtmann's bijection between the set of classical noncrossing partitions of the vertices of a regular $n$-gon and the set of maximal Hom-free sets in $\C(Q)$. In fact, Riedtmann's bijection can be stated as follows:

Let $\P = \{\B_1, \ldots, \B_m\}$ be a classical noncrossing partition of the vertices of a regular $n$-gon, and assume the elements of each $\B_i$ are in numerical order. Given $k \in \{1, \ldots, n\}$, let $\B = \{k_1, \ldots, k_s \}$ be the block that contains $k$. So $k = k_r$ for some $1 \leq r \leq s$. We associate to $k$ the indecomposable object $M_k$ of $\C(Q)$ given by $(k_r, k_{r+1})$, where $r+1$ is taken module $s$. Then Riedtmann's bijection maps $\P$ to the set $\{M_k \mid k = 1, \ldots, n \}$ of indecomposable objects.

So, for example, the noncrossing partition $\{ \{1, 2, 3\}, \{4, 5 \}, \{6\}\}$ of the hexagon corresponds to the set of indecomposable objects $\{(1,2), (2,3), (3,1), (4,5), (5,4), (6,6)\}$ in $\C(Q)$. By \ref{readingHoms}, it is clear that it is maximal Hom-free. 
\end{remark}

Using \ref{Rectangles} and the AR-formula, we have the following:

\begin{lemma}\label{ExtsLemma}
Let $X, Y$ be indecomposable objects in $\C(Q)$. Then:
\begin{enumerate}
\item $\Ext_{\C(Q)} ([i,j], Y) \ne 0$ if and only if $C(j-1, s(Y), i-1)$ and $C(i, t(Y), j-1)$.
\item $\Ext_{\C(Q)} (X, [i,j]) \ne 0$ if and only if $C(i+1, s(X), j)$ and $C(j+1, t(X), i+1)$.
\end{enumerate}
\end{lemma}

This lemma can be reinterpreted geometrically as follows:

\begin{prop}
Let $X, Y$ be indecomposable objects in $\C(Q)$. Then $\Ext_{\C(Q)} (X,Y) \ne 0$ if and only if one of the following conditions hold:
\begin{enumerate}
\item $X$ and $Y$ cross in such a way that $C(s(X), t(Y), t(X), s(Y))$, and the vertices $s(X), t(Y), t(X)$ and $s(Y)$ are pairwise distinct.
\item $t(Y)= s(X)$ and $C(t(X), s(Y), s(X)-1)$.
\item $s(Y) = t(X)$ and $C(s(X), t(Y), t(X) -1)$.
\item $s(Y) = t(X) -1$, $t(Y) \ne t(X)$ and $X, Y$ don't cross. 
\end{enumerate}
\begin{figure}[!ht]
\psfragscanon
\psfrag{a}{$X$}
\psfrag{b}{$Y$}
\includegraphics[scale=.8]{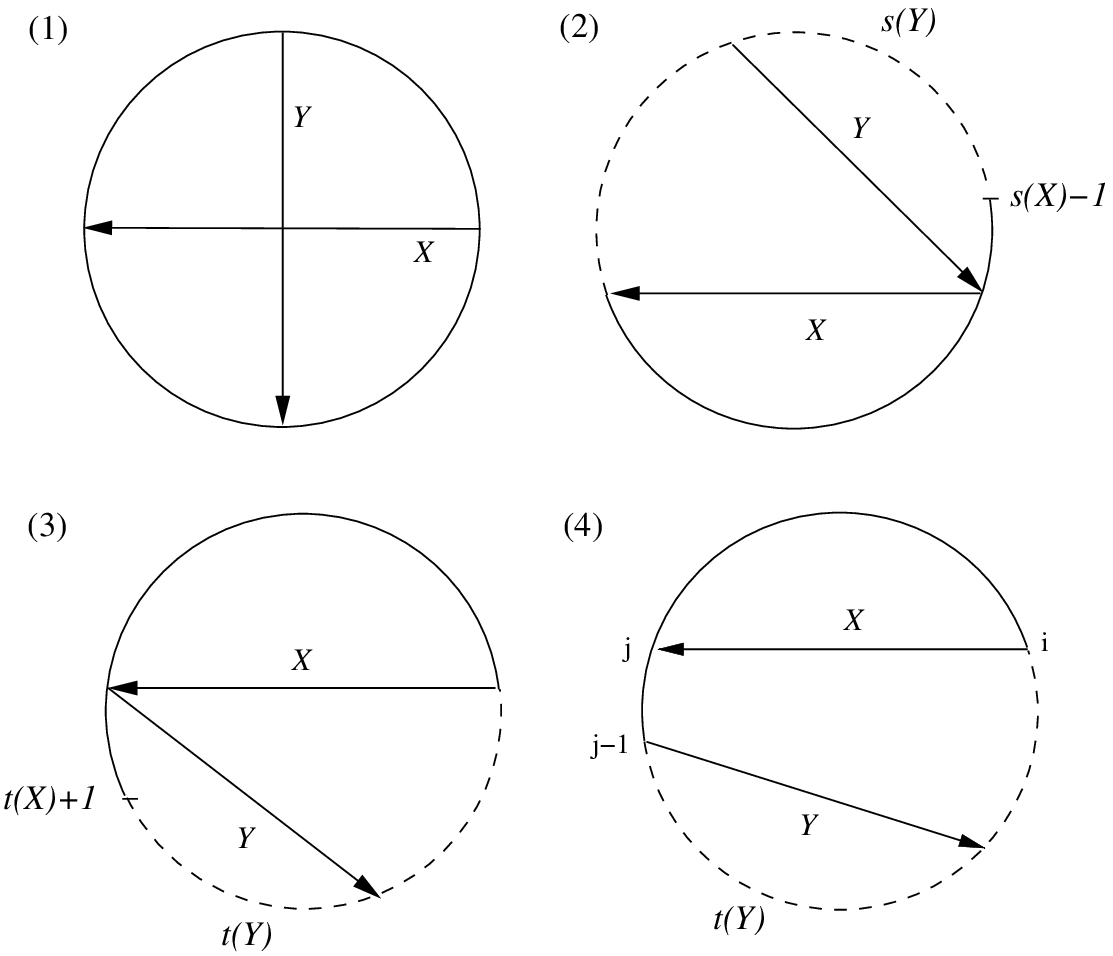}
\caption{$\Ext_{\C(Q)} (X,Y) \ne 0$.}
\end{figure}
\end{prop}

A basic object $X$ in $\C(Q)$ is called \textit{rigid} or \textit{Ext-free} if $\Ext_{\C(Q)} (X,X) =0$.

\begin{cor}\label{readingExts}
Let $X$ and $Y$ be indecomposable objects of $\C(Q)$. We have that $X \oplus Y$ is not rigid if and only if one of the following conditions holds:
\begin{enumerate}
\item $X$ and $Y$ cross each other,
\item ($t(Y) = s(X) + 1$ and $C(t(Y), s(Y), t(X))$), or ($s(Y) = t(X) -1$ and $C(s(X), t(Y),  s(Y))$). 
\item $t(X) = s(Y)$, or $s(X) = t(Y)$, where $X, Y$ are arrows which aren't loops. 
\end{enumerate}
\end{cor}

Particular cases:

\begin{remark}\label{particularcases}
\begin{enumerate}
\item If $X = [i,i]$, then 
\[
\Ext_{\C(Q)} (X,Y) \ne 0 \text{ if and only if } s(Y) = i-1, and
\]
\[
\Ext_{\C(Q)}  (Y,X) \ne 0 \text{ if and only if } t(Y) = i+1. 
 \]
\item If $X= [i,i+1]$, for some $i$, then
\[
\Ext_{\C(Q)} (X,Y) \ne 0 \text{ if and only if } t(Y) = i, and
\]
\[
\Ext_{\C(Q)} (Y,X) \ne 0 \text{ if and only if } s(Y) = i+1.
\]
\end{enumerate}
\end{remark}

\section{Classification of maximal rigid objects in $\C(Q)$}\label{secclass}

The aim in this section is to characterize the maximal rigid objects of $\C(Q)$ using the geometric model given above. 

We can view a maximal rigid object $T$ as a graph whose vertices are the vertices of $\P_n$ and the arrows are the indecomposable summands of $T$. We will determine the properties that characterize these graphs. The first one is given in the following remark. 

\begin{remark}\label{bipartite}
By \ref{readingExts} (3), we have that given a maximal rigid object $T$, there can't be a vertex $i$ in $\P_n$ with $i = s(\alpha) = t(\beta)$, where $\alpha, \beta$ are summands of $T$, $\alpha \ne \beta$, unless either $\alpha$ or $\beta$ is a loop. Hence, $T \setminus \{loops\}$ can be seen as a bipartite graph, which is noncrossing by \ref{readingExts} (1), where vertices are divided into sources and sinks. Note also that this graph might have isolated vertices (see Figure \ref{fig4} for an example).
\end{remark}

From now on, we shall tacitly switch between interpreting a maximal rigid object as a direct sum of indecomposable objects in $\C(Q)$ and as a (noncrossing, bipartite) graph whose vertices are those of $\P_n$.

\begin{figure}[!ht]
\psfragscanon
\includegraphics[scale=.8]{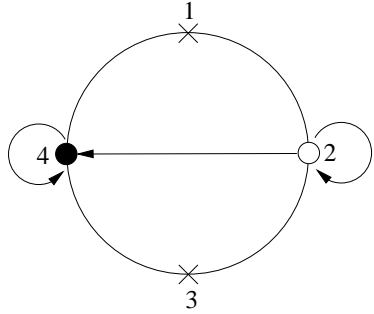}
\caption{Maximal rigid object with isolated vertices, for $n=4$.}
\label{fig4}
\end{figure}

The following proposition gives a necessary and sufficient condition for a vertex to be isolated in any given maximal rigid object.

\begin{prop}\label{isolated}
Let $T$ be a maximal rigid object in $\C(Q)$ and $i$ a vertex of $\P_n$. Then $i$ is an isolated vertex if and only if $[i+1,i+1]$ and $[i-1,i-1]$ are summands of $T$.
\end{prop}
\begin{proof}
If $[i-1,i-1]$ is a summand of $T$ then, by \ref{particularcases} (1), $i$ is either a source or isolated. Similarly, if $[i+1,i+1]$ is a summand of $T$, then $i$ is either a sink or isolated. Hence, if both these loops are summands of $T$, $i$ must be an isolated vertex. 

Conversely, suppose $i$ is an isolated vertex. We will only prove that $[i-1,i-1]$ is a summand of $T$, as similar arguments can be used to prove that $[i+1,i+1]$ is also a summand of $T$.

Suppose for a contradiction, that there is no loop at $i-1$. 

\textbf{Case 1:}  $i-1$ is either a source or an isolated vertex.

Then $[i-1,i] \not\in T$, as $i$ is isolated, and $[i-1,i] \oplus T$ is rigid, by \ref{particularcases} (2). This contradicts the maximality of $T$. 

\textbf{Case 2:} $i-1$ is a sink. 

Let $\{i_1, i_2, \ldots, i_k\}$ be the set of vertices of $\P_n$ which are the starting points of arrows in $T$ ending at $i-1$. Assume this set is ordered anticlockwise starting from $i-1$ (see Figure \ref{fig5}). 

\begin{figure}[!ht]
\psfragscanon
\psfrag{i-1}{$i-1$}
\psfrag{i}{$i$}
\psfrag{X}{$X$}
\psfrag{ik}{$i_k$}
\psfrag{i2}{$i_2$}
\psfrag{i1}{$i_1$}
\includegraphics[scale=.8]{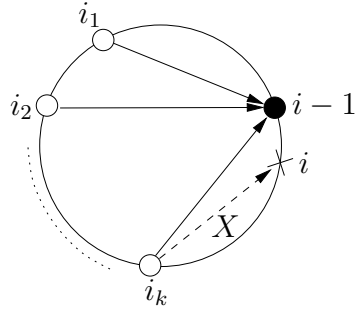}
\caption{$i \in I$ and $i-1 \in \Si$.}
\label{fig5}
\end{figure}

Consider the arrow $X = [i_k,i]$, which is not in $T$, as $i$ is isolated. Note that there is no summand $Y$ of $T$ satisfying conditions (1) or (3) in \ref{readingExts}. Suppose there is a summand $Y$ of $T$ such that $t(Y) = i_k+1$ or $s(Y) =i-1$.

Since, by assumption, there is no loop at $i-1$, and $i-1$ is a sink, we have that $s(Y) =i-1$ can't happen. But if $t(Y) = i_k +1$, then $\Ext_{\C(Q)} (Y,[i_k,i-1]) \ne 0$ (see \ref{ExtsLemma}), which contradicts the fact that $T$ is rigid. 

Therefore, $X \oplus T$ is rigid, contradicting the maximality of $T$ . Thus there must be a loop at $i-1$.
\end{proof}

\begin{remark}\label{loopneighbours}
\begin{enumerate}
\item It follows from \ref{particularcases} and \ref{isolated} that, given a maximal rigid object $T$, $T$ has a loop at $i$ if and only if $i+1$ is a source with no loop or an isolated vertex and $i-1$ is a sink with no loop or an isolated vertex. 
\item Let $T$ be a maximal rigid object, $i$ a source and $i+1$ a sink, and suppose there are no loops at $i$ and $i+1$. Then it follows from \ref{readingExts} that $[i,i+1]$ is a summand of $T$. 
\item If $i \in \So$, $i+2 \in \Si$ and $i+1 \in \I$, then it also follows from \ref{readingExts} that $[i,i+2]$ is a summand of $T$.
\end{enumerate}
\end{remark}

Given a maximal rigid object $T$ in $\C(Q)$ with $n = |Q_0|$ even, we say that $T$ has \textit{alternating loops} if every other vertex has a loop. This means, by \ref{isolated}, that the other vertices are isolated. 

\begin{prop}\label{connected}
Let $T$ be a maximal rigid object. The full subgraph of $T$ whose vertices are the set of non-isolated vertices is connected. 
\end{prop}
\begin{proof}
Suppose, for a contradiction, that the full subgraph is disconnected. Then we can divide the set of vertices of $\P_n$ into two disjoint subsets $C_1$ and $C_2$, which do not consist only of isolated vertices, in such a way that there are no arrows between $C_1$ and $C_2$. 

Given that no two arrows of $T$ cross, we can order the vertices of $C_1$ (respectively, $C_2$) as $1, \ldots, k$ (respectively, $k+1, \ldots, n$).

 
Assume, without loss of generality, that $1$ is not isolated.

We then have to consider four cases: when $1$ is a sink or a source, with or without loop. We will only check the two cases when $1$ is a sink, as the remaining cases can be checked using similar arguments.

\textbf{Case 1:} $1$ is a sink with no loop. 

Given that $1$ is a sink, there can't be a loop at $n$, by \ref{loopneighbours} (1). On the other hand, since there is no loop at $1$, $n$ is not an isolated vertex, by \ref{isolated}.

\textbf{Subcase 1.1:} $n$ is a source.

Given that there are no loops at $1$ or at $n$, $[n, 1]$ must be a summand of $T$ by \ref{loopneighbours} (2), which contradicts the hypothesis that $C_1$ and $C_2$ are not connected.

\textbf{Subcase 1.2:} $n$ is a sink. 

Let $j = min \{ l \in C_2 \, \mid \, [l, n] \in T \}$ (see Figure \ref{fig8}). 

\begin{figure}[!ht]
\psfragscanon
\psfrag{i1}{$1$}
\psfrag{in}{$n$}
\psfrag{C1}{$C_1$}
\psfrag{C2}{$C_2$}
\includegraphics[scale=.8]{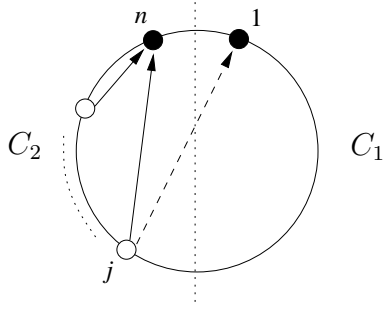}
\caption{$1$ a sink with no loop, $n$ a sink.}
\label{fig8}
\end{figure}

Consider the arrow $[j,1]$, which is not in $T$. Then there is an arrow $[a, b]$ in $T$ such that $[a,b] \oplus [j,1]$ is not Ext-free, since $T$ is maximal rigid. It follows from \ref{readingExts} that $[a,b]$ must satisfy one of the following conditions:
\begin{enumerate}[(i)] 
\item $[a,b]$ crosses $[j,1]$. 

If $[a,b]$ is of the form $[t,n]$ with $t \in C_2$ and $t < j$, this contradicts the choice of $j$. If $t \in C_1$, this contradicts the fact that $C_1$ and $C_2$ are not connected. The remaining possibilities are that $a= n$ and $1 < b < j$ or that $[a,b]$ crosses $[j,n]$. Either way, we would have $[a,b] \oplus [j,n]$ not rigid, contradicting the rigidity of $T$.

\item $[a,b] = [n, b]$, with $j \leq b \leq n$.

This can't happen as $n$ is a sink with no loop.

\item $[a,b] = [a, j+1]$, with $j+1 \leq a \leq n$ or $a = 1$.

If $j < n-1$, then $[a, j+1] \oplus [j,n]$ would not be Ext-free, a contradiction as $T$ is rigid.

Suppose then that $j = n-1$. Note that $a \not\in C_1$, otherwise $C_1$ and $C_2$ would be connected. Hence $a \in C_2$ and since $a \ne j$ and there is no loop at $n$, we have $a < j$, which contradicts the choice of $j$. 

\item $a = 1$ or $b = j$ and $a \ne b$.

This can't happen as $1$ is a sink and $j$ is a source.
\end{enumerate}

Hence, we have a contradiction in subcase 1.2. 

\textbf{Case 2:} $1$ is a sink with a loop.

Then $n$ is either a sink with no loop or an isolated vertex, by \ref{loopneighbours} (1).

\textbf{Subcase 2.1:} $n$ is a sink with no loop.

Then we can apply the same argument as in Subcase 1.2, and conclude that $T \oplus [j,1]$ is rigid, a contradiction.

\textbf{Subcase 2.2:} $n$ is isolated. 

By \ref{isolated}, there is a loop at $n-1$, and so in particular, $n-1$ is not isolated. Note that this vertex lies in $C_2$, otherwise $C_2$ would consist only of one isolated vertex, which contradicts the hypothesis.

\textbf{2.2.1:} $n-1$ is a source with loop.

By \ref{loopneighbours} (3), $[n-1, 1]$ must be a summand of $T$, and so $C_1$ and $C_2$ are connected, a contradiction. 

\textbf{2.2.2:} $n-1$ is a sink with loop. 

Take $j = min \{ l \in C_2 \, \mid \, [l, n-1] \in T \}$ (see Figure \ref{fig9}). Note that if $j = n-1$, this means that there are no arrows with $n-1$ as target other than the loop, and so $[n-1, 1] \oplus T$ would be rigid, contradicting the maximality of $T$. Therefore, $j \ne n-1$. 

\begin{figure}[!ht]
\psfragscanon
\psfrag{i1}{$1$}
\psfrag{C1}{$C_1$}
\psfrag{C2}{$C_2$}
\psfrag{in}{$n$}
\psfrag{in-1}{$n-1$}
\includegraphics[scale=.8]{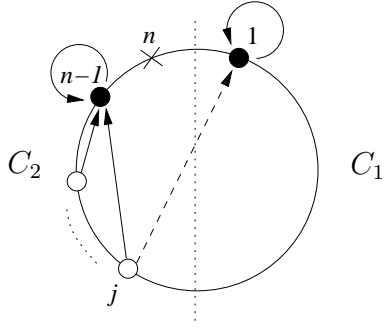}
\caption{$1$ and $n-1$ sinks with loops.}
\label{fig9}
\end{figure}

Consider the arrow $[j, 1]$ which is not in $T$ and suppose, for a contradiction, that there is an arrow $[a,b]$ in $T$ such that $[a,b] \oplus [j, 1]$ is not rigid. By \ref{readingExts}, $[a,b]$ must satisfy one of the following conditions:

\begin{enumerate}[(i)]
\item $[a,b]$ crosses $[j,1]$.

Given that $n$ is isolated and $n-1$ is a sink, $[a,b]$ either crosses $[j, n-1]$ or is of the form $[a, n-1]$ with $a \in C_1$ or $a \in C_2$ with $a < j$. The first case can't happen as $T$ is rigid, and the last two cases can't happen either as it would contradict the hypothesis or the definition of $j$.
\item $[a,b] = [n, b]$, with $j \leq b \leq n$.

Since $n$ is isolated, this case doesn't occur.
\item $[a,b] = [a, j+1]$, with $j+1 \leq a \leq n$ or $a = 1$.

Given that there is a loop at $n-1$, the vertex $n-2$ can't be a source, by \ref{loopneighbours} (1). Therefore $j \ne n-2$. But then $[a, j+1] \oplus [j, n-1]$ is not rigid, a contradiction.
\item $a = 1$ or $b = j$ and $a \ne b$.

This can't happen as $1$ is a sink and $j$ is a source. 
\end{enumerate}
 
Hence $T \oplus [j,1]$ is rigid, contradicting the maximality of $T$.

\end{proof}

Using this result we can give a greatest lower bound for the number of summands of a maximal rigid object.

\begin{remark} 
The minimum number of indecomposable summands of a maximal rigid object in $\C(Q)$ is $n-1$, where $n$ is the number of vertices of $Q$.
\end{remark}
\begin{proof}
Let $T$ be a maximal rigid object and suppose $T$ has $k < n$ isolated vertices. Given that the full subquiver of $T$ whose vertices are non-isolated is connected, it must have at least $n-k-1$ arrows. Each isolated vertex gives rise to $2$ loops by \ref{isolated}, but different isolated vertices might give rise to a common loop. The minimum number of loops is then $k$ when $T$ has alternating isolated vertices. So, in total, the number of summands of $T$ is at least $(n-k-1)+k = n-1$. 
\end{proof}

\subsection{Tilings}


Given a maximal rigid object $T$, we have seen that $T$ without the loops and the isolated vertices can be seen as a connected noncrossing bipartite graph, where the vertices are either sources or sinks (see \ref{bipartite} and \ref{connected}).

So $T \setminus \{\text{loops}\}$ can be interpreted as a polygon dissection of $\P_n$. Each polygon in this dissection shall have the name \textit{tile}. In this subsection we will give the possible tiles appearing in these graphs. 

\begin{definition}
We say that an object $T^\prime$ of $\C(Q)$ is \textit{loop-free maximal rigid} if none of its indecomposable direct summands is a loop, and there are no arrows $[i,j]$, with $i, j$ distinct non-isolated vertices, such that $T^\prime \oplus [i,j]$ is rigid.
\end{definition}

\begin{remark}\label{loopfreemaximalrigid}
Let $T$ be a maximal rigid object and let $T^\prime$ be $T \setminus \{loops\}$. Then $T^\prime$ is a loop-free maximal rigid object in $\C(Q)$.
\end{remark}
\begin{proof}
Suppose that there is an arrow $[i,j]$, with $i, j \not\in I$ and $i \ne j$, which is not a summand of $T^\prime$ such that $T^\prime \oplus [i,j]$ is rigid. 

Since $T$ is maximal rigid, we have that $T \oplus [i,j]$ is not rigid. Since all the non-loop summands of $T$ are also summands of $T^\prime$, there is a loop $[k,k]$ in $T$ such that $[k,k] \oplus [i,j]$ is not rigid. By  \ref{loopneighbours} (1), this means that either $i = k-1$ or $j = k+1$. Suppose, without loss of generality, that $i = k-1$. By hypothesis $i \not \in \I$ and since $[k,k]$ is a summand of $T$, we have that $i$ is a sink with no loop in $T$, by \ref{loopneighbours} (1). So there is an arrow $[t, i]$, with $t \ne i$, in $T^\prime$. But $[t,i] \oplus [i,j]$ is not rigid, by \ref{readingExts} (3). Hence, $T^\prime \oplus [i,j]$ is not rigid, a contradiction.
\end{proof}

Given a bipartite graph, we say that a subset of its arrows is a \textit{minimal cycle} if the induced subgraph is an (unoriented) cycle.

\begin{lemma}\label{cycles}
Let $T^\prime$ be a loop-free maximal rigid object. 
\begin{enumerate}
\item No cycle in $T^\prime$ contains an arrow of the form $i \rightarrow i+1$.
\item Any minimal cycle in $T^\prime$ has length four.
\end{enumerate}
\end{lemma}
\begin{proof}
(1) If $[i,i+1] \in T^\prime$, it follows from \ref{readingExts} (2) and \ref{particularcases} (2) that either $v(i) = 1$ or $v(i+1) = 1$, which doesn't happen if $[i,i+1]$ lies in a cycle.

(2) Suppose, for a contradiction, that there is a minimal cycle in $T^\prime$ with length $>4$. 

Pick a source in the cycle, label it by $i_1$ and label the remaining vertices in the cycle clockwise by $i_2, i_3, \ldots, i_k$ ($k>4$).

\begin{figure}[!ht]
\psfragscanon
\psfrag{i1}{$i_1$}
\psfrag{i2}{$i_2$}
\psfrag{i3}{$i_3$}
\psfrag{i4}{$i_4$}
\psfrag{i5}{$i_5$}
\psfrag{i6}{$i_6$}
\includegraphics[scale=.8]{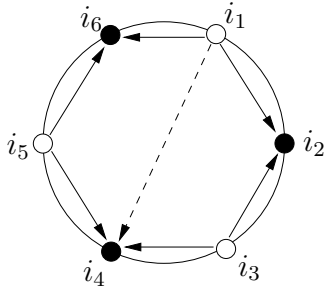}
\caption{Cycle of length 6.}
\label{fig11}
\end{figure}

Suppose, for a contradiction, that there is an arrow $[a,b]$ in $T^\prime$ such that $[i_1, i_4] \oplus [a,b]$ is not rigid. Then, by \ref{readingExts}, $[a,b]$ must satisfy one of the following conditions:
\begin{enumerate}[(i)]
\item $[a,b]$ crosses $[i_1, i_4]$.

Then either $[a,b]$ crosses one of the arrows in the cycle, which is a contradiction as $T^\prime$ is rigid, or $[a,b]$ is an arrow lying inside the cycle, contradicting the minimality of this cycle.
\item $a = i_4 -1$ and $b \ne i_4$. 

Note that $a \ne i_3$ by (1). But then $[i_4-1,b] \oplus [i_3,i_4]$ is not Ext-free, a contradiction.
\item $b = i_1 + 1$ and $a \ne i_1$.

Note that $b \ne i_2$ by (1). But then $[a,i_1 + 1] \oplus [i_1, i_2]$ is not Ext-free, a contradiction.
\item $a = i_4$ or $b = i_1$ and $a \ne b$.

This can't happen since $i_4$ is a sink and $i_1$ is a source.
\end{enumerate}
Therefore $[i_1,i_4] \oplus T^\prime$ is rigid. Given that the cycle is minimal, $[i_1,i_4] \not\in T^\prime$, which contradicts the maximality of $T^\prime$. 
\end{proof}

\begin{remark}
From now on we will assume $n$ to be greater or equal to $3$. For $n = 1$ or $2$, it is clear that there is a one-to-one correspondence between maximal rigid objects in $\C(Q)$ and indecomposable objects in $\C(Q)$. The classification we will obtain in terms of tilings only holds for $n \geq 3$.
\end{remark}

\begin{definition}
Let $\sfG$ be a connected noncrossing bipartite graph in $\P_n$, $\T$ a tile of $\sfG$ and $i_1, i_2, \ldots, i_k$ be $k$ consecutive vertices of $\P_n$, with $k \geq 2$, that lie in $\T$. If $k = 2$, we say that $\T$ has an \textit{open boundary} $(i_1,i_2)$ if there are no arrows between $i_1$ and $i_2$. If $k >2$, we say that $\T$ has an \textit{open boundary} $(i_1, i_k)$ if there are no arrows between any pair of vertices in the set $\{i_1, \ldots, i_k\}$, except possibly between $i_1$ and $i_k$. 

We define the \textit{length of the open boundary} $(i_1, i_k)$, with $k \geq 2$, to be $k-1$.
\end{definition}

\begin{prop}\label{tiles}
Let $n \geq 3$ and $T^\prime$ a loop-free maximal rigid object in $\C(Q)$. The possible tiles of $T^\prime$ are described in Figure \ref{fig14}.
\begin{figure}[!ht]
\psfragscanon
\psfrag{C1}{$C_1$}
\psfrag{C2}{$C_2$}
\psfrag{E1}{$E_1$}
\psfrag{E2}{$E_2$}
\includegraphics[scale=.8]{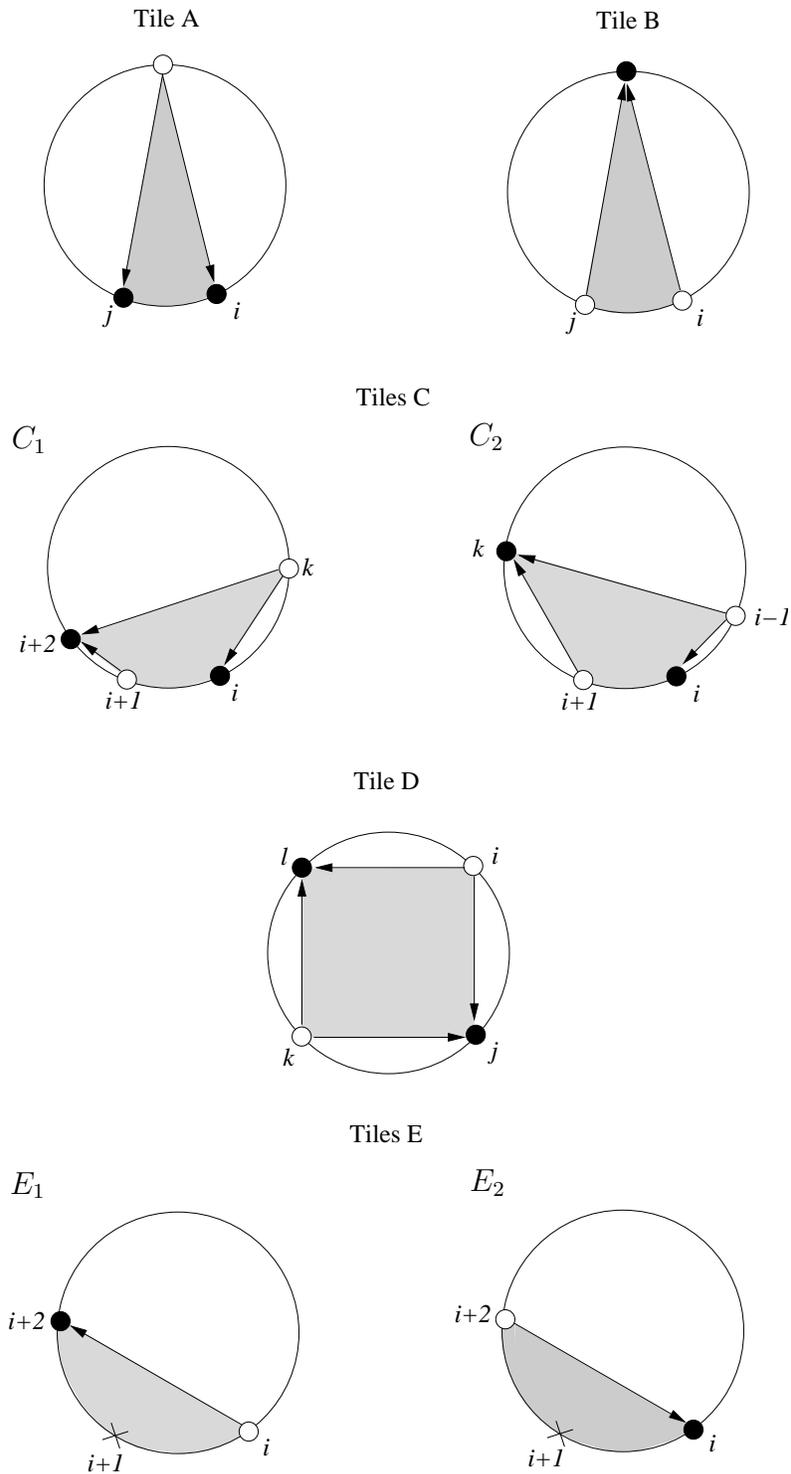}
\caption{Tiles that can appear in a maximal rigid object.}
\label{fig14}
\end{figure}

Note that in tiles A and B, $j = i+1$ or $j = i+2$, and in tile D we must have $j \ne i+1$ and $l \ne k+1$. Note also that all the vertices in Figure \ref{fig14} are distinct from each other.
\end{prop}
\begin{proof}
It follows from \ref{cycles} and its proof that any cycle in $T^\prime$ bounds a region which is a union of tiles, each of which is bounded by a 4-cycle, and these 4-cycles can't have an arrow of the form $[i,i+1]$. 

Let $\T$ be a tile of $T^\prime$ which is not bounded by a 4-cycle. Then it must have an open boundary. Note that only one part of the boundary of $\T$ can be an open boundary, by the connectedness of $T^\prime$. Let $(i,j)$ be the open boundary of $\T$. If the length of $(i,j)$ is strictly greater than two, then there is at least one non-isolated vertex $k$ in the open boundary of $\T$, since isolated vertices cannot be adjacent. But then there must be an arrow incident with $k$. This arrow either crosses another arrow of $T^\prime$ or lies in the interior of the tile $\T$. The former contradicts the noncrossing property of $T^\prime$ and the latter contradicts the definition of tile. Hence $j$ is either $i+1$ or $i+2$, and if $j = i+2$ then $i+1$ must be isolated. 


\textbf{Case 1:} $i$ and $j$ are sources. 

Consider the arrows $[i,l]$ and $[j,k]$ in the tile $\T$ starting at $i$ and $j$ respectively (see Figure \ref{fig17}). 

\begin{figure}[!ht]
\psfragscanon
\psfrag{T}{$\T$}
\includegraphics[scale=.8]{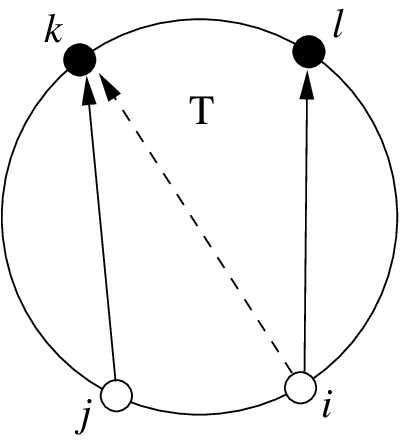}
\caption{Tile $\T$ with $i,j \in \So$.}
\label{fig17}
\end{figure}

Suppose, for a contradiction, that $k \ne l$. Then $[i,k] \not\in T^\prime$. However, using similar arguments as the ones used in the proof of \ref{cycles} (2), one can see that $T^\prime \oplus [i,k]$ is rigid, which contradicts the (loop-free) maximality of $T^\prime$. So $k = l$, and $\T$ is of type B. 

\textbf{Case 2:} $i$ and $j$ are sinks.

Using a similar argument as in case 1, we can conclude that $\T$ must be of type A.

\textbf{Case 3:} $i$ is a source and $j$ is a sink.

Then it follows from \ref{loopneighbours} (2) or (3) that $[i,j]$ must be a summand of $T^\prime$. Thus, given that $(i,j)$ is an open boundary, we must have $j = i+2$ and $i+1 \in \I$. So $\T$ is a tile of type $E_1$.

The remaining case is:

\textbf{Case 4:} $i$ is a sink and $j$ is a source.

If $[j,i] \in T^\prime$ then $j = i+2$ and $i+1 \in \I$, and so $\T$ is of type $E_2$. 

Otherwise, let $[l,i]$ and $[j,k]$ be the arrows in the tile $\T$ ending at $i$ and starting at $j$ respectively (see Figure \ref{fig18}).

\begin{figure}[!ht]
\psfragscanon
\psfrag{T}{$\T$}
\includegraphics[scale=.8]{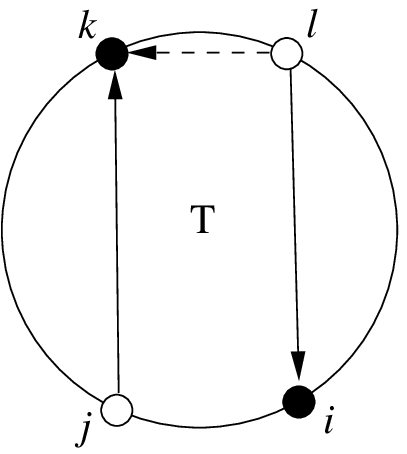}
\caption{Tile $\T$ with $i\in \Si$ and $j \in \So$.}
\label{fig18}
\end{figure}

Note that $[l,k] \oplus T^\prime$ is rigid (using similar arguments as the ones used in the proof of \ref{cycles}, for example), and so, since $T^\prime$ is maximal, we have $[l,k] \in T^\prime$. Suppose $k \ne j+1$ and $l \ne i-1$. Then $[j,i] \oplus T^\prime$ is rigid, which contradicts the maximality of $T^\prime$. Therefore $k= j+1$ or $l= i-1$. Now we check that $j = i+1$. Suppose for a contradiction, that $j = i+2$, and so $i +1 \in \I$. Hence there must be a loop at $i+2$, which implies that $i+3$ is either isolated or a source. If we are in the case when $k = j+1 = i+3$, this does not happen as $k$ is a sink. So we must have $l=i-1$. But we also have a loop at $i$, which implies that $i-1$ is either isolated or a sink, a contradiction as $l$ is a source. Hence $j = i+1$ and so the tile $\T$ in this case is of type C. 
\end{proof}

\subsection{Isolated vertices}

We will now focus our attention on the isolated vertices. The following remark, which follows immediately from \ref{isolated} and \ref{loopneighbours} (1), gives necessary conditions for a vertex to be isolated using the neighboring vertices rather than loops.

\begin{remark}\label{isolatedneighbors}
Let $T$ be a maximal rigid object with an isolated vertex $i$. Then the following conditions must hold in $T \setminus \{\text{ loops }\}$:
\begin{enumerate}
\item $i-2$ is either a sink or isolated. 
\item $i+2$ is either a source or isolated.
\item $i-1, i+1, i-3$ and $i+3$ can't be isolated vertices.
\end{enumerate}
\end{remark}

From now on, a \textit{tiling} of $\P_n$ means a collection of tiles of type A, B, C, D and E (see \ref{tiles}) glued to each other in such a way that they cover the polygon $\P_n$ and the isolated vertices satisfy the three conditions in \ref{isolatedneighbors}.

\begin{cor}\label{loop-free maximal implies tiling}
Let $n \geq 3$. Any loop-free maximal rigid object is a tiling of $\P_n$.
\end{cor}
\begin{proof}
It follows immediately from \ref{tiles} and \ref{isolatedneighbors}.
\end{proof}

The following proposition claims that the converse of \ref{loop-free maximal implies tiling} also holds. 

\begin{prop} \label{loopfreemaximalrigidtiling}
Let $T^\prime$ be a tiling of $\P_n$. Then $T^\prime$ is loop-free maximal rigid.
\end{prop}
\begin{proof}
Since, by definition, the tiling respects \ref{readingExts} (1) and (3), we only need to check \ref{readingExts} (2) for rigidity. Let $i$ be a sink and $i-1$ a source. Note that $[i-1,i]$ is an arrow in the tiling $T^\prime$. We need to check that $v(i) = 1$ or $v(i-1) = 1$. In order to check this, we look at the (unique) tile incident with $[i-1,i]$.

\textbf{Case 1:} Tile A: Here we have $v(i) = 1$ (see Figure \ref{fig21}).

\begin{figure}[!ht]
\psfragscanon
\includegraphics[scale=.8]{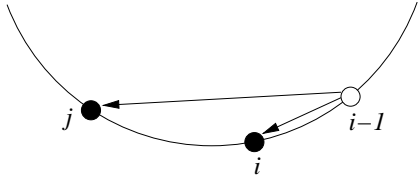}
\caption{Tile A incident with $[i-1,i]$, $j = i+1$ or $j= i+2$.}
\label{fig21}
\end{figure}

\textbf{Case 2:} Tile B: Here we have $v(i-1) = 1$ (see Figure \ref{fig22}).

\begin{figure}[!ht]
\psfragscanon
\includegraphics[scale=.8]{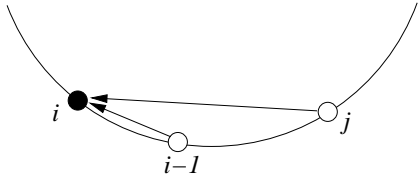}
\caption{Tile B incident with $[i-1,i]$, $j=i-2$ or $j= i-3$.}
\label{fig22}
\end{figure}

\textbf{Case 3:} Tile C: We have $v(i-1) = 1$ in the case $C_1$ and $v(i) = 1$ in the case $C_2$ (see Figure \ref{fig23}).

\begin{figure}[!ht]
\psfragscanon
\psfrag{C1}{$C_1$}
\psfrag{C2}{$C_2$}
\includegraphics[scale=.8]{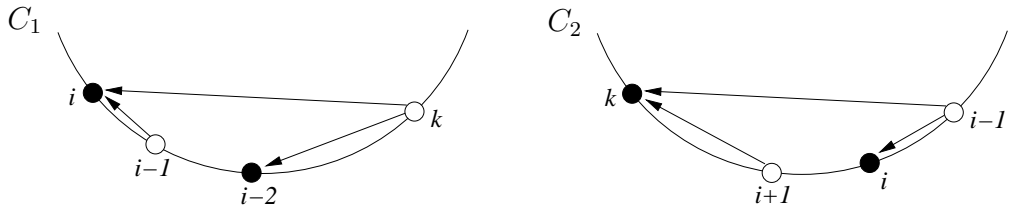}
\caption{Tile C incident with $[i-1,i]$.}
\label{fig23}
\end{figure}

Recall that clockwise boundary edges don't appear in tiles of type D or E, so we are done.

To check the maximality, note that the only arrows we can add to the tiling, while preserving the noncrossing rule and without changing whether a vertex is a sink, a source or isolated, are the open boundaries in tiles of type C. But then we would have a 4-cycle with a clockwise boundary edge, which is not rigid, by \ref{cycles} (1). Therefore, we can't add any arrow to the tiling.
\end{proof}

\subsection{Loops} 

Given a tiling (recall that the three conditions in \ref{isolatedneighbors} are satisfied), we want to add loops in such a way that we get a maximal rigid object. The way to do this is given in the following proposition.

\begin{prop}\label{addloops}
Let $T$ be a tiling of $\P_n$, and let $T_L$ be the directed graph obtained from $T$ by adding loops in the following manner:

$(L_1)$ If $i \not\in \I$, add a loop at $i$ if and only if $i-1 \in \Si \cup \I$, $i+1 \in \So \cup \I$ and $i-2, i+2 \not\in \I$. 

$(L_2)$ Remove one of the loops in the situation of Figure \ref{fig24}.

\begin{figure}[!ht]
\psfragscanon
\includegraphics[scale=.8]{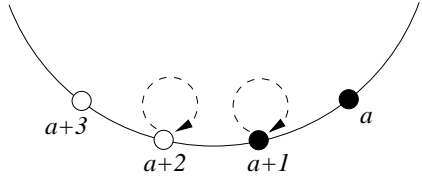}
\caption{Unique case where there is a choice of loops.}
\label{fig24}
\end{figure}

Then $T_L$ corresponds to a maximal rigid object of $\C(Q)$. 
\end{prop}
\begin{proof}
Suppose that $T_L$ is not rigid. By \ref{loopfreemaximalrigidtiling} we have that $T$ is rigid. Hence there is a loop $[i,i]$ and an arrow $[a,b]$ in $T_L$ such that $[i,i] \oplus [a,b]$ is not rigid. By \ref{particularcases} (1), $a = i-1$ or $b= i+1$. But by $(L_1)$, $i-1 \in \Si \cup \I$ and $i+1 \in \So \cup \I$. So we must have $a = b = i-1$ or $a = b = i+1$. Suppose, without loss of generality, that $a = b = i-1$. 

Given that we are not adding loops at isolated vertices, we must have $i \in \So$ and $i-1 \in \Si$, by $L_1$. It also follows by $L_1$ that $i+1$ can't be isolated, since there is a loop at $i-1$, and so $i+1$ must be a source, as $i$ has a loop. Analogously, $i-2$ must be a sink. 

So we are in the situation of Figure \ref{fig24}, and $L_2$ tells us that one of these two loops is not in $T_L$, a contradiction. Therefore $T_L$ is rigid.

To check the maximality, we will first check that we can't add more loops while preserving the Ext-free property.

Suppose we can add a loop at an isolated vertex $a$. By the conditions in \ref{isolatedneighbors}, we have that $a-1 \not\in \I$, $a-2 \in \Si \cup \I$ and $a-3, a+1 \not\in \I$. Hence, by $L_1$, there is a loop at $a-1$. Note that, since $a$ is isolated, we are not in the situation of Figure \ref{fig24}, and so the loop at $a-1$ is not removed, i.e., $[a-1,a-1] \in T_L$. However, $[a,a] \oplus [a+1,a+1]$ is not Ext-free, a contradiction.

Suppose now we can add a loop at a non-isolated vertex $a$. We want to check that this loop is already in $T_L$, i.e., $a$ satisfies the conditions in $L_1$ and not the ones in $L_2$. 

By \ref{particularcases} (1), $a+1 \in \So \cup \I$, $a-1 \in \Si \cup \I$ and there are no loops at these two vertices. Suppose, for a contradiction, that $a-2 \in \I$. Then by the conditions in \ref{isolatedneighbors}, we have $a \in \So$ (as we assumed $a$ is not isolated), and $a+1, a-1, a-3 \not\in \I$. Hence, by $L_1$, there is a loop at $a-1$ in $T_L$ (note that this loop is not removed as we are not in the situation of $L_2$ as $a-2 \in \I$). We have reached a contradiction, and so $a-2 \not\in \I$. Analogously, $a+2 \not\in \I$. 

So $a$ satisfies conditions in $L_1$ and since there are no loops at $a-1$ and $a+1$, we are not in the situation described in $L_2$. Therefore $[a,a] \in T_L$. 

All that remains to be checked is that no non-loop arrows can be added to $T_L$ while preserving the Ext-free property. It follows immediately from \ref{loopfreemaximalrigidtiling} that no arrows incident to non-isolated vertices can be added. In order to prove that one can't add an arrow incident to an isolated vertex, we will first check that if $a$ is isolated then $L_1$ and $L_2$ tell us that $[a+1,a+1], [a-1,a-1] \in T_L$. 

Indeed, suppose that the vertex $a$ of $T_L$ is isolated. Since the conditions in \ref{isolatedneighbors} are satisfied, we have that $a+2 \in \So \cup \I$, and $a+1, a+3, a-1 \not\in \I$. Hence, by $L_1$, $[a+1, a+1] \in T_L$. Note that $L_2$ can't be applied to this loop as $a$ is isolated. We can similarly prove that $[a-1,a-1] \in T_L$. 

Now it follows from \ref{particularcases} (1) that an arrow incident to $a$ would have an extension with one of these two loops, and so no arrows incident to isolated vertices can be added. 
%
%
%
%
\end{proof}

\begin{prop}\label{mroimpliesl1l2}
Let $n \geq 3$ and $T$ a maximal rigid object in $\C(Q)$. Then the loops are obtained from the tiling, i.e., from $T \setminus \{loops\}$, using $L_1$ and $L_2$.
\end{prop}
\begin{proof}
Let $T^\prime = T \setminus \{ loops \}$, and let $T_1$ be $T^\prime$ together with the loops obtained by applying $L_1$ to $T^\prime$.  

Firstly, we prove that $T \subseteq T_1$. Let $[i,i]$ be a loop in $T$. Then, in particular, $i$ can't be an isolated vertex. By \ref{loopneighbours} (1), we have that $i+1 \in \So \cup \I$, $i-1 \in \Si \cup \I$ and there are no loops at these two vertices. So, in particular, by \ref{isolated}, $i+2, i-2 \not\in \I$. Hence, $i$ satisfies the conditions in $L_1$, and so $[i,i] \in T_1$. This proves that $T \subseteq T_1$.

As we have seen in \ref{addloops}, $L_1$ can give rise to consecutive loops in the situation of Figure \ref{fig24}. Given that $T$ is rigid, we must have $T \subseteq T_1 (c)$, where $T_1 (c)$ is obtained from $T_1$ by removing some loops according to $L_2$, for some choice $c$. By \ref{addloops}, $T_1 (c)$ is rigid, and so, since $T$ is maximal rigid, we have $T = T_1(c)$, which finishes the proof.
\end{proof}

By putting all these results together, we obtain a characterization of the maximal rigid objects in $\C(Q)$, which is stated in the following theorem.

\begin{thm}
Let $n \geq 3$. There is a one-to-one correspondence between maximal rigid objects in $\C(Q)$ and tilings of $\P_n$ together with loops obtained by performing the operations described in $L_1$ and $L_2$.
\end{thm}
\begin{proof}
Let $T$ be a maximal rigid object in $\C(Q)$. Then $T \setminus \{loops\}$ is a tiling of $\P_n$ by \ref{tiles} and \ref{isolatedneighbors}. By \ref{mroimpliesl1l2}, the loops are obtained by performing the operations in $L_1$ and $L_2$. The converse follows from \ref{addloops}.  
%
%
\end{proof}


\subsection{Example.} In this subsection we will see that there is a relationship between some maximal rigid objects in $\C(Q)$ and combinatorial objects studied in \cite{Noy}.\\

Let $n > 2$ and consider a bipartite noncrossing tree $T$ on a sequence of $r$ consecutive sources followed by a chain of $n-r$ consecutive sinks on a circle. Label the sources clockwise by $1, \ldots, r$ and the sinks by $r+1, \ldots, n$.

Note that $T$ can be viewed as a tiling of $\P_n$ using just tiles of types A and B, and with no isolated vertices. By applying $L_1$ and $L_2$ of \ref{addloops} we get a maximal rigid object, whose loops are:
\begin{enumerate}
\item at vertex $i_n$ if $r = 1$ (and so $n-r >1$).
\item at vertex $i_1$ if $n-r = 1$ (and so $r > 1$).
\item either at vertex $i_1$ or at vertex $i_n$ if $r, n-r >1$. 
\end{enumerate}

These maximal rigid objects are in 1-1 correspondence with the sections of the AR-quiver of $\C(Q)$ which don't contain either of the pairs $[i,i+1], [i,i]$, or $[i,i+1], [i+1,i+1]$, simultaneously. 

Indeed, let $a$ be a source, i.e., $a \in \{i_1, \ldots, i_r\}$, and let $b^a = min \{ x \, \mid \, x \in \Si \text{ and } [a,x] \in T \}$. Note that $[a, c] \in T$, for all $b^a \leq c \leq b^{a-1}$. Hence, the set of arrows of $T$ is
\[
\{[a, b^a], [a,b^a+1], \ldots, [a, b^{a-1}] \, \mid \, a \in \So, \, a \ne i_1 \} \cup \{[i_1,b^{i_1}], [i_1,b^{i_1}+1], \ldots, [i_1,i_n] \}.
\]
These arrows together with one of the loops described above form a section in the AR-quiver of $\C(Q)$. 

Note that the sections that contain $[i,i]$ and $[i,i+1]$ (or $[i,i+1]$ and $[i+1,i+1]$) simultaneously, for some $i$, don't correspond to maximal rigid objects, since $[i,i+1] \oplus [i,i]$ (respectively, $[i,i+1] \oplus [i+1,i+1]$) is not rigid.

Figure \ref{fig42} shows an example for $n = 5$ (there is a choice between the two possible loops, which are dashed).

\begin{figure}[!ht]
\psfragscanon
\includegraphics[scale=.8]{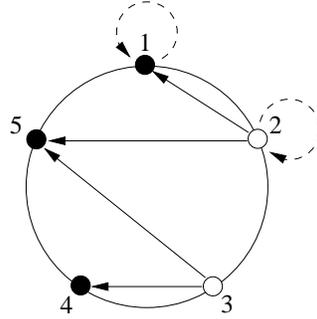}
\caption{Bipartite noncrossing tree with 2 consecutive sources followed by 3 consecutive sinks on a circle.}
\label{fig42}
\end{figure}

The corresponding sections (with a choice between the objects which are overlined) in the AR-quiver of the orbit category $\C(Q)$ is as follows:

\[
\xymatrix@C=0.1cm{& & & & \overline{11} \ar@{->}[dr] & & \overline{22} \ar@{->}[dr] & & 33 \ar@{->}[dr] & & 44 \ar@{->}[dr] & & 55 \ar@{->}[dr] & & 11 \\
& & & 15 \ar@{->}[dr]\ar@{->}[ur] & & \underline{21} \ar@{->}[dr]\ar@{->}[ur] & & 32 \ar@{->}[dr]\ar@{->}[ur] & & 43 \ar@{->}[dr]\ar@{->}[ur] & & 54 \ar@{->}[dr]\ar@{->}[ur] & & 15 \ar@{->}[ur] \\
& & 14 \ar@{->}[dr]\ar@{->}[ur] & & \underline{25} \ar@{->}[dr]\ar@{->}[ur] & & 31 \ar@{->}[dr]\ar@{->}[ur] & & 42 \ar@{->}[dr]\ar@{->}[ur] & & 53 \ar@{->}[dr]\ar@{->}[ur] & & 14 \ar@{->}[ur]\\
& 13 \ar@{->}[dr]\ar@{->}[ur] & & 24 \ar@{->}[dr]\ar@{->}[ur] & & \underline{35} \ar@{->}[dr]\ar@{->}[ur] & & 41 \ar@{->}[dr]\ar@{->}[ur] & & 52 \ar@{->}[dr]\ar@{->}[ur] & & 13 \ar@{->}[ur]\\
12 \ar@{->}[ur] & & 23 \ar@{->}[ur] & & \underline{34} \ar@{->}[ur] & & 45 \ar@{->}[ur] & & 51 \ar@{->}[ur] & & 12 \ar@{->}[ur]}
\]

Bipartite noncrossing trees $T$ on a chain of $r$ consecutive sources followed by a chain of $s = n-r$ consecutive sinks on a circle were counted by Mark Noy (cf. \cite[Theorem 4.1]{Noy}), and the number of them is given by 
\[
\binom{r+s-2}{r-1}.
\]



\section{The endomorphism algebras - quivers and relations}\label{secendalg}

Given a maximal rigid object $T$ in $\C(Q)$, we will now describe the endomorphism algebra $\End_{\C(Q)} (T)$ in terms of quivers with relations, using the combinatorial characterization of $T$.

We will assume $n$ to be greater or equal to $3$. Note that if $n = 1$ or $2$, $\End_{\C(Q)} (T)$ is given by a quiver with one vertex and no loops.

Let $\T$ be the tiling of the $n$-gon $\P_n$ corresponding to $T$. We define a quiver $Q_{\T}$ associated to $\T$ as follows:

\textbf{vertices of $Q_{\T}$}: The vertices correspond to all the arrows (including loops) of $\T$. 

\textbf{arrows of $Q_{\T}$}: Two vertices of $Q_{\T}$, which correspond to non-loops, are related by an edge in $Q_{\T}$ if the corresponding arrows of $\T$ share a vertex and belong to the same tile. 

Orientation: Let $\alpha, \beta$ be two arrows of $\T$ sharing a vertex $x$ of $\P_n$ and belonging to the same tile. We say that $\alpha < \beta$ if the rotation with minimal angle around $x$ that sends $\alpha$ to $\beta$ is clockwise (see Figure \ref{fig26}).

\begin{figure}[!ht]
\psfragscanon
\psfrag{x}{$x$}
\psfrag{a}{$\alpha$}
\psfrag{b}{$\beta$}
\includegraphics[scale=.8]{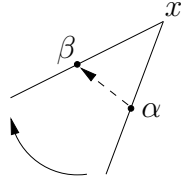}
\caption{$\alpha < \beta$.}
\label{fig26}
\end{figure}

The edge in $Q_{\T}$ between $\alpha$ and $\beta$ is oriented $\alpha \rightarrow \beta$ if $\alpha < \beta$.

\medskip

Let $l$ be the vertex in $Q_{\T}$ associated to a loop at a vertex $x$ of $\P_n$. Let $\{ \alpha_1, \ldots, \alpha_k \}$ be the set of arrows (excluding loops) in $\T$ incident with $x$, and suppose this set is ordered clockwise.


If $x$ is a source, then there is an arrow in $Q_{\T}$ from the vertex associated to $\alpha_k$ to $l$. If $x$ is a sink, then there is an arrow in $Q_{\T}$ from $l$ to the vertex associated to $\alpha_1$ (see Figure \ref{fig28}). 

\begin{figure}[!ht]
\psfragscanon
\psfrag{x}{$x$}
\psfrag{a1}{$\alpha_1$}
\psfrag{ai}{$\alpha_i$}
\psfrag{ak}{$\alpha_k$}
\includegraphics[scale=.8]{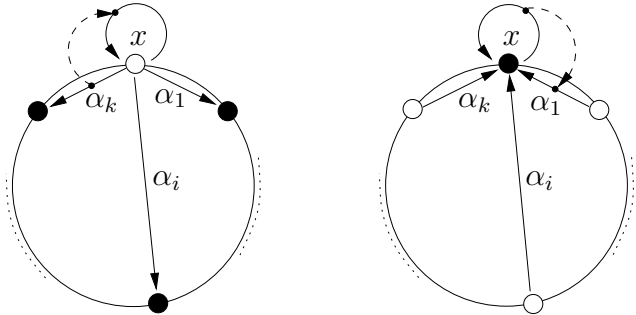}
\caption{Arrows in $Q_{\T}$ incident with vertices corresponding to loops.}
\label{fig28}
\end{figure}

\textbf{Relations $\R_{\T}$ in $Q_{\T}$}: In any tile, the composition of two sucessive arrows is zero. Here, we consider the loop at $x$ to be in the boundary tile incident with $\alpha_k$ if $x$ is a source, or in the boundary tile incident with $\alpha_1$ if $x$ is a sink. 

We denote by $I_{\T}$ the ideal generated by the relations $R_{\T}$.

\begin{ex}
Figure \ref{fig36} provides some examples when $n = 6$.
\begin{figure}[!ht]
\psfragscanon
\includegraphics[scale=.8]{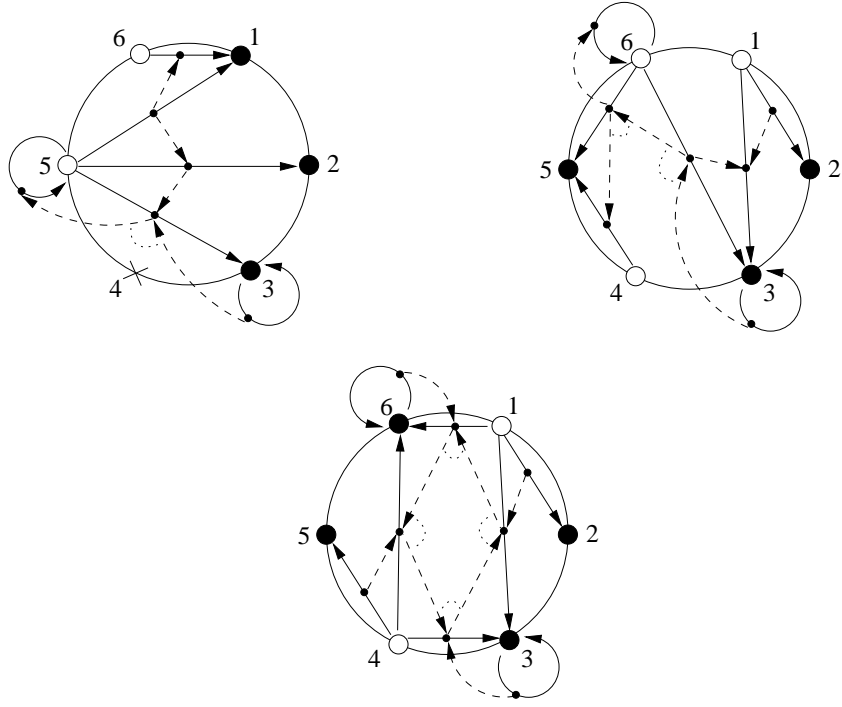}
\caption{Some illustrations of $(Q_\T, I_\T)$ for $n = 6$.}
\label{fig36}
\end{figure}
\end{ex}

\begin{thm}
The endomorphism algebra $\End_{\C(Q)} (T)$ is isomorphic to the path algebra given by $(Q_{\T}, I_{\T})$.
\end{thm}
\begin{proof}
Write $T = \oplus T_i$, where $T_i$ are the indecomposable summands of $T$, pairwise non-isomorphic. Since $\End_{\C(Q)} (T_i) = K$, for all $i$, we have no loops in the quiver of the endomorphism algebra.  

Fix $i$ and write $\overline{T} = \oplus_{j \ne i} T_j$. Let $T_i \rightarrow E$ be the minimal right $\add (\overline{T})$ - approximation of $T_i$.  The number of arrows in the quiver of the endomorphism algebra from the vertex $i$ to $j$, with $j \ne i$, is given by the multiplicity of $T_j$ in E. 

Note that, if $E \ne 0$, then $\Hom_{\C(Q)} (T_i,E) \ne 0$, which implies that all summands of $E$ lie in $\R_F (T_i)$. But, if $E \in \R_F (T_i) \setminus \sfS(T_i)$, where $\sfS (T_i)$ is the section corresponding to $T_i$, then we would have $\Hom_{\C(Q)} (\tau^{-1} T_i, E) \simeq \Ext_{\C(Q)} (E,T_i) \ne 0$. This contradicts the fact that $E \in \add (\overline{T})$ and $T_i \oplus \overline{T}$ is rigid. Hence, $E \in \sfS (T_i)$.

Let $T_i = [a,b]$ and let us denote by $\sfS^{T_i}$ the full subquiver of $\sfS (T_i)$ whose objects are of the form $[a,x]$ with $C(b, x, a)$. Similarly, we denote by $\sfS_{T_i}$ the full subquiver of $\sfS (T_i)$ whose objects are of the form $[y, b]$ with $C(a, y, b-1)$. 

We have that $E = E_1 \oplus E_2$, where $E_1$ and $E_2$ are as follows:

$E_1 = [a,x^\prime]$, where $x^\prime = min \{ x \, \mid \, b < x \leq a \text{ and } [a,x] \in \overline{T} \}$. Otherwise, i.e., if $\overline{T}$ doesn't have any summand in $\sfS^{T_i}$, then $E_1 = 0$.

Analogously, $E_2 = [y^\prime,b]$, where $y^\prime = min \{ y \, \mid \, a < y \leq b-1 \text{ and } [y,b] \in \overline{T}\}$. Otherwise, i.e., if $\overline{T}$ doesn't have any summand in $\sfS_{T_i}$, then $E_2 = 0$.

Suppose $T_i$ is not a loop, i.e., $a \ne b$. If $x^\prime \ne a$, then note that $T_i$ and $E_1$ lie in the same tile and $T_i < E_1$, so there is an arrow from the vertex associated to $T_i$ to the vertex associated to $E_1$ in $Q_{\T}$. The same happens with $T_i$ and $E_2$ (see Figure \ref{fig30}).

\begin{figure}[!ht]
\psfragscanon
\psfrag{a}{$a$}
\psfrag{b}{$b$}
\psfrag{j}{$x^\prime$}
\psfrag{l}{$y^\prime$}
\psfrag{E1}{$E_1$}
\psfrag{E2}{$E_2$}
\psfrag{Ti}{$T_i$}
\includegraphics[scale=.8]{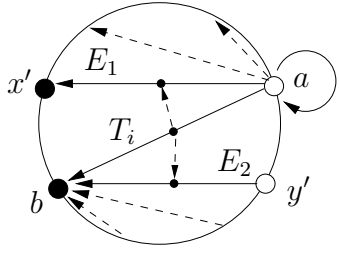}
\caption{$T_i \rightarrow E_1$ and $T_i \rightarrow E_2$.}
\label{fig30}
\end{figure}

If $x^\prime = a$, then there is an arrow from the vertex associated to $T_i$ to the vertex associated to the loop $E_1 =[a,a]$ in $Q_{\T}$ ($T_i$ plays the role of $\alpha_k$ and $E_1$ plays the role of $l$ in the definition of $Q_{\T}$). 

Now, if $a=b$, i.e., if $T_i$ is a loop, then $E_1 = 0$, as $\sfS^{T_i}$ has just one vertex, which is $T_i$ itself. Moreover, there is an arrow from $T_i$ to $E_2$ in $Q_{\T}$ ($T_i$ plays the role of $l$ and $E_2$ plays the role of $\alpha_1$ in the definition of $Q_{\T}$).  

We have found all the arrows starting at the vertex corresponding to $T_i$ (for each $i$) in the quiver of the endomorphism algebra and checked that these are the same as those in the quiver $Q_{\T}$. This proves that the quiver of the endomorphism algebra is $Q_{\T}$.

We will now check that the relations $R_{\T}$ are satisfied.

First, let us check that the composition of two arrows  in the same tile, whose sources and targets don't correspond to loops, is zero. Note that tiles of type A and B give rise to just one arrow in $Q_{\T}$. On the other hand, the arrows in $Q_{\T}$ appearing in tiles of type E have source or target corresponding to a loop. So we just need to check the tiles of type C and D.

\textbf{Tile $C_1$}: See Figure \ref{fig31}.
\begin{figure}[!ht]
\psfragscanon
\psfrag{a}{$\alpha$}
\psfrag{b}{$\beta$}
\includegraphics[scale=.8]{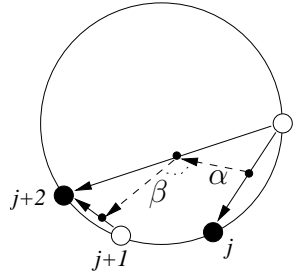}
\caption{$\beta \alpha = 0$ in tile $C_1$.}
\label{fig31}
\end{figure}

The arrows $[i,j]$ and $Y=[j+1, j+2]$ don't satisfy any of the conditions in \ref{readingHoms}, hence $\Hom_{\C(Q)} ([i,j], [j+1,j+2]) = 0$, and so $\beta \alpha = 0$. We will omit the case when the tile is of type $C_2$, as the argument is the same.

\textbf{Tile D}: See Figure \ref{fig32}.
\begin{figure}[!ht]
\psfragscanon
\psfrag{a}{$\alpha$}
\psfrag{b}{$\beta$}
\psfrag{g}{$\gamma$}
\psfrag{d}{$\delta$}
\includegraphics[scale=.8]{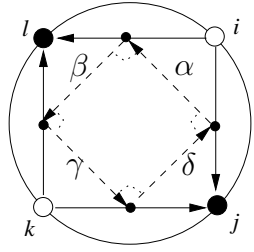}
\caption{The composition of two arrows in tile D is zero.}
\label{fig32}
\end{figure}

Note that $[i,j]$ and $[k,l]$ don't satisfy any of the conditions in \ref{readingHoms}, and so $\Hom_{\C(Q)} ([i,j], [k,l]) = 0$. Therefore $\beta \alpha = 0$. For the same reason we have $\gamma \beta = 0$, $\delta \gamma = 0$ and $\alpha \delta = 0$. 

 Now, let $l$ be a vertex of $Q_{\T}$ corresponding to a loop $[i,i]$. If $i$ is a source (respectively, a sink), we have only one arrow $\alpha$ in $Q_{\T}$ incident with $l$ and $t(\alpha) = l$ (respectively, $s(\alpha) = l$). Recall that, given $X \in \ind \, \C(Q)$, $\Hom_{\C(Q)} (X, [i,i]) \ne 0$ (respectively, $\Hom_{\C(Q)} ([i,i], X) \ne 0$) if and only if $s(X) = i$ (respectively, $t(X) = i$). There are no arrows in the same tile of $[i,i]$ with source $i$ (respectively, target $i$) other than the arrow corresponding to $s(\alpha)$ (respectively, $t(\alpha)$. Therefore, we have that $\alpha \beta = 0$, for any $\beta \in Q_{\T}$ lying in the same tile as $[i,i]$. 


All we need to check now is that all the relations in the quiver of the endomorphism algebra are generated by the relations in $\R_\T$. 

First we consider zero-relations. Let $\alpha_k \cdots \alpha_2 \alpha_1$ be a path in $Q_\T$ of length $k \geq 2$. We either have two consecutive arrows $\alpha_j, \alpha_{j-1}$, for some $j$, in the same tile or not. In the former case we have $\alpha_j \alpha_{j-1} \in \R_\T$, and so $\alpha_k \cdots \alpha_2 \alpha_1 \in I_\T$. In the latter case, if $T_i$, with $i = 1, \ldots, k-1$, denotes the indecomposable summand corresponding to $s(\alpha_i)$ and $T_k$ denotes the indecomposable summand corresponding to $t(\alpha_k)$, we must have:
\begin{enumerate}
\item $T_i = [a, b_i]$, for $i = 1, \ldots, k$ and $C(b_1, b_2, \ldots, b_k)$, or
\item $T_i = [a_i,b]$, for $i = 1, \ldots, k$ and $C(a_1, a_2, \ldots, a_k)$.
\end{enumerate}
In the first case we have $T_i \in \sfS^{T_1}$, for $i = 2, \ldots, k$ and the second case we have $T_i \in \sfS_{T_1}$, for $i = 2, \ldots, k$. Either way it is easy to see that $\alpha_k \cdots \alpha_2 \alpha_1 \ne 0$. 


Now we only need to check that there are no relations involving several paths. Let $\Sigma_{i = 1}^k \lambda_i p_i$ be a relation with $k$ minimal, $k >1$ and suppose $p_i \ne 0$, for all $i$. Note that for all $i$ we have $s(p_i) = a$ and $t(p_i) = b$, for some $a$ and $b$. Let $[i,j]$ be the summand of $T$ corresponding to the vertex $a$ of $Q_{\T}$. Note that there are at most two distinct paths, which are different than zero, starting at $a$: one which goes around $j$ and one which goes around $i$. Moreover, these two paths must finish at different vertices. Hence $k = 1$, a contradiction. It follows from this argument that the only relations in $Q_{\T}$ are zero relations, which were already covered. 
\end{proof}

\begin{remark}\label{minimalrelations}
It is easy to check that $\R_{\T}$ is the set of all the minimal relations.
\end{remark}

Note that different maximal rigid objects can give rise to the same endomorphism algebra. Figure \ref{fig41} gives an example.

\begin{figure}[!ht]
\psfragscanon
\includegraphics[scale=.8]{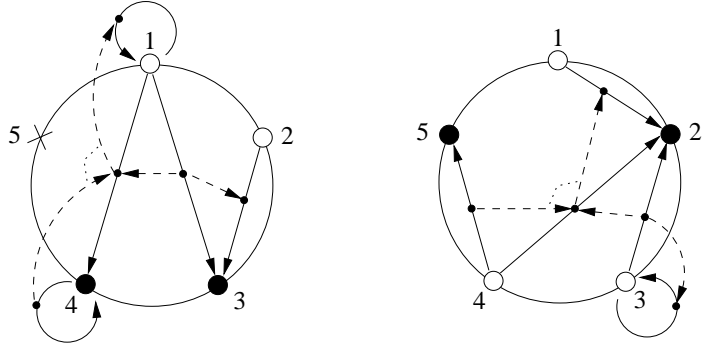}
\caption{Different maximal rigid objects can have the same endomorphism algebra.}
\label{fig41}
\end{figure}

\section{Iterated tilted algebras}\label{seciterated}

Iterated-tilted algebras were introduced in \cite{AH} and they can be defined as follows:

\begin{definition}
Let $G$ be a quiver. An algebra $B$ is called an \textit{iterated tilted algebra of type $G$} is there exists a family $(A_i,T_{A_i})$ with $0 \leq i \leq m$ consisting of algebras $A_i$ and tilting modules $T_{A_i}$ such that:
\begin{enumerate}
\item $A_0$ is hereditary with quiver $G$,
\item $A_{i+1} \simeq (\End_{A_i} \, T)^{op}$, for $0 \leq i \leq m$, and 
\item each indecomposable $A_{i+1}$-module $M$ is of the form $\Hom_{A_i} (T, N)$ or $\Ext^1_{A_i} (T, N)$ for some indecomposable $A_i$-module $N$.
\end{enumerate}
\end{definition}

Iterated tilted algebras of type $A_n$ were characterized in terms of quivers with relations by Happel. 

\begin{thm}\cite[Corollary in Section 5]{H}\label{thmHappel}
A finite dimensional algebra $A = KQ/I$ is an iterated tilted algebra of type $A_n$ if and only if:
\begin{enumerate}
\item The underlying graph $\overline{G}$ of $G$ is a tree,
\item the minimal relations have length 2,
\item every vertex has at most four neighbours,
\item if four neighbours occur, then we are in the situation of Figure \ref{fig37}, where $\alpha \, \beta = 0 = \gamma \, \delta$, is a full subquiver of $(Q,I)$,

\begin{figure}[!ht]
\psfragscanon
\psfrag{a}{$\alpha$}
\psfrag{b}{$\beta$}
\psfrag{d}{$\delta$}
\psfrag{g}{$\gamma$}
\includegraphics[scale=.8]{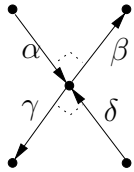}
\caption{When four neighbours occur.}
\label{fig37}
\end{figure}

\item if three neighbours occur, then

\begin{figure}[!ht]
\psfragscanon
\psfrag{a}{$\alpha$}
\psfrag{b}{$\beta$}
\includegraphics[scale=.8]{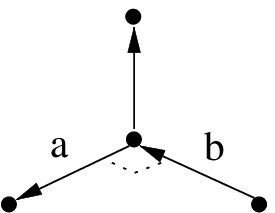}
\label{fig38}
\end{figure}

or

\begin{figure}[!ht]
\psfragscanon
\psfrag{a}{$\alpha$}
\psfrag{b}{$\beta$}
\includegraphics[scale=.8]{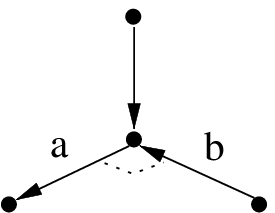}
\label{fig38a}
\end{figure}

where $\alpha \, \beta = 0$, is a full subquiver of $(Q,I)$.
\end{enumerate}
\end{thm}

In this section we will check which of the endomorphism algebras of maximal rigid objects in $\C(Q)$ are iterated tilted algebras of type $A$. 

\begin{prop}
Let $T$ be a maximal rigid object in $\C(Q)$ and $\T$ the corresponding tiling. Then the endomorphism algebra $\End_{\C(Q)} (T)$ is an iterated tilted algebra of type $A$ if and only if $T$ is a tiling with no tiles of type D.
\end{prop}
\begin{proof}
Note that if $T$ has a tile of type D, then the quiver of the corresponding endomorphism algebra has a cycle, and therefore $\End_{\C(Q)} (T)$ cannot be an iterated tilted algebra of type $A$. Now, suppose that $T$ doesn't have any tile of type D. We will check that $(Q_{\T}, \I_{\T})$ verifies every condition in \ref{thmHappel}.

(1) The underlying graph of $Q_{\T}$ is a tree:

Let $a$ be a vertex of $Q_{\T}$ and suppose that $a$ lies in a cycle 
\[
\xymatrix{a = a_1 \ar[r]^{\alpha_1} & a_2 \ar[r]^{\alpha_2} & \cdots \ar[r]^{\alpha_{m-1}} & a_m \ar[r]^{\alpha_m} & a}
\]
in $Q_{\T}$.

Let $T_{a_i}$ be the indecomposable summand of $T$ corresponding to the vertex $a_i$, for each $1 \leq i \leq m$. 

The valency of the vertices of $Q_{\T}$ corresponding to loops is one, hence no $T_{a_i}$ is a loop. The arrow $T_a = [i,j]$ is part of the boundary of at most two tiles $\T_1$ and $\T_2$ (see Figure \ref{fig39a}).

\begin{figure}[!ht]
\psfragscanon
\psfrag{T1}{$\T_1$}
\psfrag{T2}{$\T_2$}
\includegraphics[scale=.8]{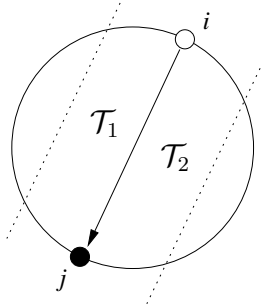}
\caption{Tiles with boundary $T_a = [i,j]$.}
\label{fig39a}
\end{figure}

Suppose, without loss of generality, that $\alpha_1$ lies in tile $\T_1$ (if $\alpha_1$ lies in $\T_2$ the argument is similar).

We have assumed that $\T_1$ is not of type D. Note that $\T_1$ cannot be of type $E_1$, by definition of the tile $E_1$ (see Figure \ref{fig14}). If $\T_1$ is of type $E_2$, then $T_{a_2}$ is a loop, which is a contradiction. Hence $\T_1$ is either of type A or C. Either way, note that $\alpha_m$ cannot lie in $\T_1$ unless $[j,j]$ is a summand of $T$ and $T_{a_m}$ is a loop, which is a contradiction. Therefore $\alpha_m$ must lie inside $\T_2$.

This means that the cycle starts in $\T_1$ and must end in $\T_2$, which is impossible. Hence, $\overline{Q_{\T}}$ is a tree.

(2) It is obvious that the minimal relations in $Q_{\T}$ have length 2 (see \ref{minimalrelations}).

(3) Every vertex has at most four neighbours:

Let $a$ be a vertex of $Q_{\T}$ and $T_a$ the corresponding indecomposable summand of $T$. If $T_a$ is a loop, then $v(a) = 1$. So suppose $T_a$ is not a loop. Then $T_a$ is part of the boundary of at most two tiles. Note that the maximum number of neighbours of a given vertex of $Q_{\T}$ which lie in the same tile is at most 2. So the valency of $a$, which lies in one or two tiles, is at most four.

(4) Let $a$ be a vertex with four neighbours. Then the corresponding indecomposable summand $T_a$ of $T$ is not a loop and $T_a$ is part of the boundary of two tiles $\T_1$ and $\T_2$. 

Since the number of neighbours of $a$ in each tile can't be more than two, we are in the situation of Figure \ref{fig40}.

\begin{figure}[!ht]
\psfragscanon
\psfrag{T1}{$\T_1$}
\psfrag{T2}{$\T_2$}
\psfrag{v1}{$v_1$}
\psfrag{v2}{$v_2$}
\psfrag{v3}{$v_3$}
\psfrag{v4}{$v_4$}
\psfrag{i}{$i$}
\psfrag{j}{$j$}
\psfrag{a}{$a$}
\includegraphics[scale=.8]{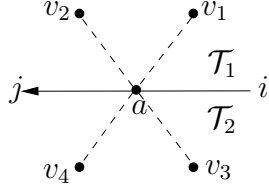}
\caption{Vertex $a$ with four neighbours.}
\label{fig40}
\end{figure}

Note that arrows of $Q_{\T}$ in the same tile are oriented, so write $\beta: a \rightarrow v_1, \alpha: v_2 \rightarrow a, \delta: v_3 \rightarrow a,$ and $\gamma: a \rightarrow v_4$.

Given that $\alpha$ and $\beta$ (respectively, $\delta$ and $\gamma$) are in the same tile, $\alpha \, \beta$ (respectively, $\gamma \, \delta$) is a relation. Note that $\alpha \, \delta$ and $\gamma \, \beta$ are not in $\I_{\T}$ since these are compositions of arrows in different tiles.

(5) One can use similar arguments to the ones used above to prove this condition. 
\end{proof}

It is natural to ask if every iterated tilted algebra of type $A_n$ is the endomorphism algebra of a maximal rigid object in $\C(Q)$. The answer is no and the following remark provides a counter-example.


\begin{remark}
There is no maximal rigid object in $\C(Q)$ for which the corresponding endomorphism algebra is given by the quiver with relations in Figure \ref{figiteratednotendo}.  

\begin{figure}[!ht]
\psfragscanon
\psfrag{a}{$\alpha$}
\psfrag{b}{$\beta$}
\psfrag{g}{$\gamma$}
\psfrag{d}{$\delta$}
\includegraphics[scale=.8]{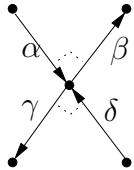}
\caption{Iterated tilted algebra which is not an endomorphism algebra of a maximal rigid object in $\C(Q)$.}
\label{figiteratednotendo}
\end{figure}
\end{remark}
\begin{proof}
Suppose there is a maximal rigid object $T$ for which the quiver with relations above is its endomorphism algebra. Let $[i,j]$ be the indecomposable summand of $T$ corresponding to the vertex $a$ with four neighbours. Note that $[i,j]$ divides the disc $\P_n$ into two parts $P_1$ and $P_2$, such that the vertices of $P_1$ (respectively, $P_2$) are $\{x \, \mid \, C(j,x,i)\}$ (respectively, $\{ y \, \mid \, C(i,y,j)\}$). The arrows $\alpha, \beta$ must lie in the same tile, say $\T_1$ and $\gamma, \delta$ lie in a different tile $\T_2$. Assume, without loss of generality that $\T_1$ lives in $P_1$ and $\T_2$ lives in $P_2$ (see Figure \ref{fig43}).  

\begin{figure}[!ht]
\psfragscanon
\psfrag{a}{$\alpha$}
\psfrag{b}{$\beta$}
\psfrag{g}{$\gamma$}
\psfrag{d}{$\delta$}
\psfrag{T1}{$\T_1$}
\psfrag{T2}{$\T_2$}
\psfrag{i}{$i$}
\psfrag{j}{$j$}
\includegraphics[scale=.8]{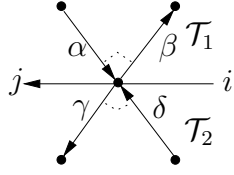}
\caption{$C(i, s(\delta), t(\gamma), j, s(\alpha), t(\beta))$.}
\label{fig43}
\end{figure}

Note that $\T_1$ cannot be of type D, otherwise the quiver $Q_T$ would have more arrows than the quiver in Figure \ref{figiteratednotendo}. 

Suppose $\T_1$ is of type A (respectively, B). Then $\T_1$ is incident with three vertices in $P_1$: $i, j$ and a third one $k$. We must have a loop at $j$ (respectively, $i$) in order to get the two arrows $\alpha$ and $\beta$ in $Q_T$.  But, by \ref{loopneighbours} (1), a loop at $j$ (respectively, $i$) implies that the open boundary of $\T_1$ has an isolated vertex, as $k \in \Si$ (respectively, $k \in \So$). Consequently there is a loop at vertex $k$, by \ref{isolated}. This loop gives rise to another arrow in $Q_T$ which do not exist in the quiver of Figure \ref{figiteratednotendo}. 

Suppose $\T_1$ is of type C. Because $\T_1$ lives in $P_1$, the open boundary of this tile is either $(i-1,i)$ or $(j,j+1)$ (see Figure \ref{fig44}). But then $Q_T$ is not the quiver in Figure \ref{figiteratednotendo}. 

\begin{figure}[!ht]
\psfragscanon
\includegraphics[scale=.8]{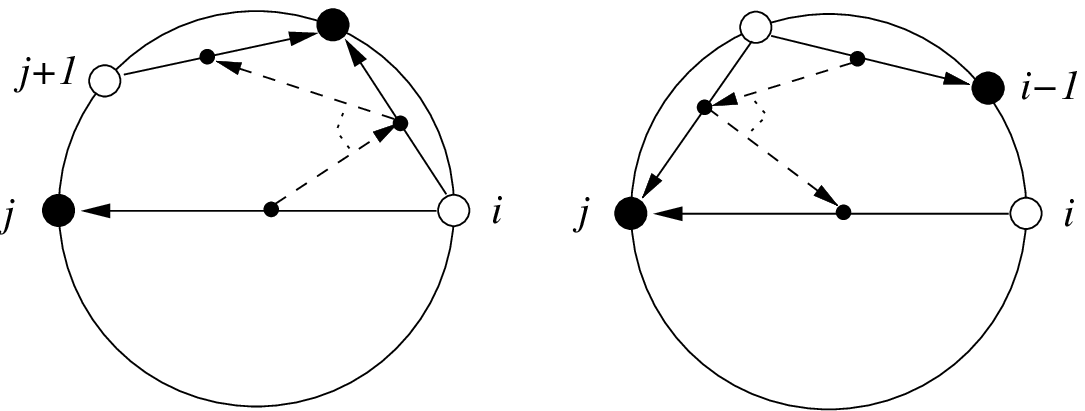}
\caption{$\T_1$ of type C.}
\label{fig44}
\end{figure}

Because $\T_1$ lives in $P_1$, $\T_1$ cannot be of type $E_1$. Therefore, the only remaining case is when $\T_1$ is of type $E_2$. In this case we have $j = i-2$ and $i-1 \in \I$. So there are loops at $i$ and $j$ and $s(\alpha)$ is the vertex at $[j,j]$ and $t(\beta)$ is the vertex at $[i,i]$. We haven't reached any contradiction yet, so we have to look at tile $\T_2$.

Tiles of type A or B just give rise to one arrow incident to vertex $a$, and the tile of type $E_1$ doesn't give rise to any more arrows in $Q_{\T}$. $\T_2$ cannot be of type $E_2$ as $\T_2$ lives in $\P_2$, and clearly $\T_2$ cannot be of type D either. Therefore, $\T_2$ must be of type C. Because $\T_2$ lives in $\P_2$, the open boundary $(k,k+1)$ (note that $k \in \Si$ and $k+1 \in \So$) of $\T_2$ is opposite $[i,j]$ and we must have $k+1 = j-1 = i-3$ or $k = i+1$. But, by \ref{loopneighbours} (1), the loops at $i$ and $j$ impose that $i-3 \in \Si \cup \I$ and $i+1 \in \So \cup \I$, a contradiction. This finishes the proof. 
\end{proof}

\subsection*{Acknowledgments.}
The author would like to express her gratitude to her supervisor, Robert Marsh, for his help and advice. She would also like to thank Fundac\~ao para a Ci\^encia e Tecnologia, for their financial support through Grant SFRH/ BD/ 35812/ 2007.


\begin{thebibliography}{30}
\bibitem{htt}
Angeleri H\"{u}gel, L., Happel, D., Krause, H.,
\emph{Handbook of Tilting Theory}, 
London Math. Soc. Lecture Note Ser., \textbf{332}, Cambridge University Press, 2007.
\bibitem{AH}
Assem, I., Happel, D., 
Generalized tilted algebras of type $A_n$, 
\emph{Comm. Algebra} \textbf{9} (20) (1981), 2101--2125.
\bibitem{A}
Amiot, C.,
On the structure of triangulated categories with finitely many indecomposables,
\emph{Bull. Soc. Math. France} \textbf{135} (2007), no. 3, 435--474.
\bibitem{BaurMarsh}
Baur, K., Marsh, R.,
Categorification of a frieze pattern determinant, arXiv:1008.5329.
\bibitem{BM}
Buan, A., Marsh, R.,
Cluster-tilting theory, in: Trends in representation theory of algebras and related topics, in \emph{Contemp. Math.}, \textbf{406}, Amer. Math. Soc., Providence, RI, 2006, 1--30. 
\bibitem{BMRRT}
Buan, A., Marsh, R., Reineke, M., Reiten, I., Todorov, G., 
Tilting theory and cluster combinatorics, \emph{Adv. Math.} \textbf{204} (2006), no. 2, 572--618.
\bibitem{BRT}
Buan, A., Reiten, I., Thomas, H.,
From m-clusters to m-noncrossing partitions via exceptional sequences, to appear in Math. Zeit.
\bibitem{CCS}
Caldero, P., Chapoton, F., Schiffler, R., 
Quivers with relations arising from clusters ($A_n$ case), \emph{Trans. Amer. Math. Soc.} \textbf{358} , no. 3, (2006) 1347--1364.
\bibitem{H}
Happel, D.,
Tilting sets on cylinders, \emph{Proc. London Math. Soc.} (3) \textbf{51} (1985), no. 1, 21--55.
\bibitem{Happelcorrected}
Happel, D.,
Corrigendum: ``Tilting sets on cylinders'', \emph{Proc. London Math. Soc.} (3) \textbf{56} (1988), no. 2, 260.
\bibitem{Kellersurvey}
Keller, B., 
Cluster algebras, quiver representations and triangulated categories,
\emph{London Math. Soc. Lecture Note Ser.}, \textbf{375}, Cambridge Univ. Press, Cambridge, 2010, 76-–160.
\bibitem{Keller}
Keller, B., 
On triangulated orbit categories, \emph{Doc. Math.} \textbf{10} (2005), 551--581.
\bibitem{Noy}
Noy, M.,
Enumeration of noncrossing trees on a circle, \emph{Discrete Mathematics}, \textbf{180} (1998), 301--313.
\bibitem{Reitensurvey}
Reiten, I., Tilting theory and cluster algebras, to appear in Proc. Trieste Workshop.
\bibitem{Riedtmann}
Riedtmann, C.; Representation-finite selfinjective algebras of class $A_n$, Representation theory II (Prof. Second Internat. Conf., Carleton Univ., Ottawa, Ont. 1979) 449--520, \emph{Lecture Notes in Math.} \textbf{832}, Springer, Berlin, 1980.
\bibitem{S}
Sim\~oes, R. C., 
Hom-configurations and noncrossing partitions,
\emph{J. Algebraic Combin.}, published \emph{Online First},
DOI: 10.1007/s10801-011-0305-5, 6 August 2011.
\end{thebibliography}
\end{document}